\documentclass[10pt]{article} 
\usepackage[accepted]{tmlr}

\def\month{05}  
\def\year{2023} 
\def\openreview{\url{https://openreview.net/forum?id=jWr41htaB3}} 
\usepackage[utf8]{inputenc}

\title{A Stochastic Proximal Polyak Step Size}

\author{\name Fabian Schaipp \email fabian.schaipp@tum.de \\
      \addr Department of Mathematics\\
      Technical University of Munich
      \AND
      \name Robert M.\ Gower \email rgower@flatironinstitute.org \\
      \addr Center for Computational Mathematics\\
      Flatiron Institute, New York
      \AND
      \name Michael Ulbrich \email m.ulbrich@tum.de\\
      \addr Department of Mathematics\\
      Technical University of Munich}

\usepackage{lipsum}
\usepackage{amsmath,amssymb,amsfonts,amsthm}
\usepackage{dsfont}
\usepackage{bm}
\usepackage{float}

\usepackage{algorithm}
\usepackage{etoolbox}
\let\classAND\AND
\let\AND\relax
\usepackage[noend]{algorithmic}

\let\AND\classAND
\AtBeginEnvironment{algorithmic}{\let\AND\algoAND}

\usepackage{enumitem}
\usepackage{xcolor}
\usepackage[hidelinks]{hyperref}
\usepackage[capitalize, nameinlink]{cleveref}
\usepackage{subcaption}
\usepackage{caption}

\usepackage{url}
\usepackage{cite}

\newtheorem{theorem}{Theorem}
\newtheorem{lemma}[theorem]{Lemma}
\newtheorem{proposition}[theorem]{Proposition}
\newtheorem{corollary}[theorem]{Corollary}

\newtheorem{assumption}{Assumption}

\newtheorem{remark}{Remark}

\definecolor{kleinblue}{RGB}{0, 47, 167}
\hypersetup{
	colorlinks = True,
	linkcolor = kleinblue,
	citecolor=kleinblue,
} 

\definecolor{todored}{RGB}{189, 30, 30}

\usepackage[colorinlistoftodos,bordercolor=orange,backgroundcolor=orange!20,linecolor=orange,textsize=scriptsize]{todonotes}

\let\temp\phi
\let\phi\varphi
\let\varphi\temp

\let\temp\varepsilon
\let\epsilon\varepsilon
\let\varepsilon\temp

\DeclareMathOperator*{\argmin}{arg\,min}

\newcommand{\prox}[1]{\mathrm{prox}_{#1}}
\newcommand{\env}[1]{\mathrm{env}_{#1}}

\newcommand{\iprod}[2]{\langle #1, #2 \rangle}

\newcommand{\oneover}[1]{\frac{1}{#1}}
\newcommand{\toneover}[1]{\tfrac{1}{#1}}
\newcommand{\T}{\top}

\newcommand{\R}{\mathbb{R}}
\newcommand{\Rn}{\mathbb{R}^n}
\newcommand{\N}{\mathbb{N}}

\newcommand{\E}{\mathbb{E}}

\newcommand{\spsmax}{\texttt{SPS}\textsubscript{max}}
\newcommand{\ie}{i.e.\ }
\begin{document}
\maketitle
\begin{abstract}
Recently, the stochastic Polyak step size (\texttt{SPS}) has emerged as a competitive adaptive step size scheme for stochastic gradient descent.  Here we develop \texttt{ProxSPS},  a \textit{proximal} variant of \texttt{SPS} that can handle regularization terms. Developing a proximal variant of \texttt{SPS} is particularly important, since \texttt{SPS} requires a lower bound of the objective function to work well. When the objective function is the sum of a loss and a regularizer, available estimates of a lower bound of the sum can be loose. In contrast, \texttt{ProxSPS} only requires a lower bound for the loss which is often readily available.
As a consequence, we show that \texttt{ProxSPS} is easier to tune and more stable in the presence of regularization. Furthermore for image classification tasks, \texttt{ProxSPS} performs as well as \texttt{AdamW} with little to no tuning, and results in a network with smaller weight parameters.
We also provide an extensive convergence analysis for \texttt{ProxSPS} that includes the non-smooth,  smooth, weakly convex and strongly convex setting.
\end{abstract}

\section{Introduction}
Consider problems of the form 
\begin{align}
    \label{prob:unreg}
    \min_{x \in \R^n} f(x), \quad f(x) := \mathbb{E}_P[f(x;S)] = \int_\mathcal{S}f(x;s)dP(s),
\end{align}
where $\mathcal{S}$ is a sample space (or sample set). Formally, we can see $S$ as a random variable mapping to $\mathcal{S}$ and $P(s)$ as the associated probability measure. 
Let us assume that for each $s\in \mathcal{S}$, the function $f(\cdot;s): \mathbb{R}^n \to \mathbb{R}$ is locally Lipschitz and hence possesses the Clarke subdifferential $\partial f(\cdot;s)$ \citep{Clarke1983}.
Problems of form \eqref{prob:unreg} arise in machine learning tasks where $\mathcal{S}$ is the space of available data points \citep{Bottou2018}. 
An efficient method for such problems is stochastic (sub)gradient descent \citep{Robbins1951, Bottou2010, Davis2019}, given by  
\begin{align}
    \tag{\texttt{SGD}}
    \label{eqn:sgd}
    x^{k+1} = x^k - \alpha_k g_k, \quad g_k \in \partial f(x^k;S_k), 
    \quad \mbox{where } S_k \sim P.
\end{align}
Moreover, we will also consider the composite problem
\begin{align}
    \label{prob:composite}
    \min_{x \in \R^n} \psi(x), \quad \psi(x) := f(x) + \phi(x),
\end{align}
where $\phi:\R^n\to\R \cup \{\infty\}$ is a proper, closed, and convex regularization function.
In practical situations, the expectation in the objective function $f$ is typically approximated by a sample average over $N\in\N$ data points. We formalize this special case with 
\begin{align}
    \tag{\texttt{ER}}
    \label{eqn:ER}
    \mathcal{S}=\{s_1,\dots,s_N\},~ P(s_i) = \frac{1}{N}, ~ f_i:= f(\cdot;s_i) \quad i=1,\dots,N.
\end{align}
In this case, problem \eqref{prob:unreg} becomes the empirical risk minimization problem 
\begin{align*}
    \min_{x \in \R^n} \frac{1}{N} \sum_{i=1}^N f_i(x).
\end{align*}
\subsection{Background and Contributions}
\textbf{Polyak step size.} For minimizing a convex, possibly non-differentiable function $f$, \citet[Chapter 5.3]{Polyak1987} proposed
\begin{align*}
    x^{k+1} = x^k - \alpha_k g_k, \quad \alpha_k = \frac{f(x^k) - \min f}{\|g_k\|^2}, \quad g_k \in \partial f(x^k)\setminus \{0\}.
\end{align*}
This particular choice of $\alpha_k$, requiring the knowledge of $\min f$, has been subsequently called the \textit{Polyak step size} for the subgradient method.  
Recently, \citet{Berrada2019, Loizou2021, Orvieto2022} adapted the Polyak step size to the stochastic setting: consider the \eqref{eqn:ER} case and assume that each $f_i$ is differentiable and that a lower bound $C(s_i) \leq \inf_x f_{i}(x)$ is known for all $i\in[N]$. The method proposed by \citep{Loizou2021} is 
\begin{align}
    \label{eqn:sps-max}
    \tag{\spsmax}
    x^{k+1} = x^k - \min\Big\{\gamma_b, \frac{f_{i_k}(x^k)-C(s_{i_k})}{c\|\nabla f_{i_k}(x^k)\|^2}\Big\}\nabla f_{i_k}(x^k),
\end{align}
with hyper-parameters $c,\gamma_b > 0$ and where in each iteration $i_k$ is drawn from $\{1,\dots,N\}$ uniformly at random. It is important to note that the initial work \citep{Loizou2021} used $C(s_i) = \inf f_i$; later, \citet{Orvieto2022}
established theory for~\eqref{eqn:sps-max} for the more general case of $C(s_i) \leq \inf_x f_i(x)$ and allowing for mini-batching. 
Other works analyzed the Polyak step size in the convex, smooth setting \citep{Hazan2019} and in the convex, smooth and stochastic setting \citep{Prazeres2021}. Further, the stochastic Polyak step size is closely related to stochastic model-based proximal point \citep{Asi2019} as well as stochastic bundle methods \citep{Paren2022}.

\emph{Contribution.} We propose a proximal version of the stochastic Polyak step size, called \texttt{ProxSPS}, which explicitly handles regularization functions. Our proposal is based crucially on the fact that the stochastic Polyak step size can be motivated with stochastic proximal point for a truncated linear model of the objective function (we explain this in detail in \cref{sec:model-based-view}).
Our method has closed-form updates for squared $\ell_2$-regularization. We provide theoretical guarantees for \texttt{ProxSPS} for any closed, proper, and convex regularization function (including indicator functions for constraints). Our main results, \cref{thm:convex-smooth-reg} and \cref{thm:exact-nonconv-reg}, also give new insights for \spsmax{}, in particular showing exact convergence for convex and non-convex settings.

\textbf{Lower bounds and regularization.} Methods such as \spsmax{} need to estimate a lower bound $C(s)$ for each loss function $f(\cdot;s)$. Though  $\inf_x f(x;s)$ can be precomputed in some restricted settings, in practice the lower bound $C(s)=0$ is used for non-negative loss functions.\footnote{See for instance~\url{https://github.com/IssamLaradji/sps}.} 
The tightness of the choice $C(s)$ is further reflected in the constant $\sigma^2:= \min f - \E_P[C(S)]$, which affects the convergence guarantees of \spsmax{} \citep{Orvieto2022}.

\emph{Contribution.} For regularized problems \eqref{prob:composite} and if $\phi$ is differentiable, the current proposal of \spsmax{} would add $\phi$ to every loss function $f(\cdot;s)$. In this case, for non-negative regularization terms, such as the squared $\ell_2$-norm, the lower bound $C(s)=0$ is always loose. Indeed, if $\phi\geq 0$, then $\inf_{x\in \R^n}(f(x;s) + \phi(x)) \geq \inf_{x\in\R^n} f(x;s)$ and this inequality is strict in most practical scenarios.
For our proposed method \texttt{ProxSPS}, we now need only estimate a lower bound for the loss $f(x;s)$ and not for the composite function $f(x;s) + \phi(x)$. Further, \texttt{ProxSPS} decouples the adaptive step size for the gradient of the loss from the regularization (we explain this in detail in \cref{sec:speciall2case} and \cref{fig:illustrations}).

\textbf{Proximal and adaptive methods.} The question on how to handle regularization terms has also been posed for other families of adaptive methods. For \texttt{Adam} \citep{Kingma2015} with $\ell_2$-regularization it has been observed that it generalizes worse and is harder to tune than  
\texttt{AdamW} \citep{Loshchilov2019} which uses weight decay.
Further, \texttt{AdamW} can be seen as an approximation to a proximal version of \texttt{Adam} \citep{Zhuang2022}.\footnote{For \texttt{SGD} treating $\ell_2$-regularization as a part of the loss can be seen to be equivalent to its proximal version (cf.\ \cref{sec:sgd-equiv}).}
On the other hand, \citet{Loizou2021} showed that -- without regularization -- default hyperparameter settings for \spsmax{} give very encouraging results on matrix factorization and image classification tasks. This is  promising since it suggests that \spsmax{} is an \emph{adaptive} method, and can work well across varied tasks without the need for extensive hyperparameter tuning.

\emph{Contribution.} 
%
%
We show that by handling $\ell_2$-regularization using a proximal step, our resulting \texttt{ProxSPS} is less sensitive to hyperparameter choice as compared to \spsmax. 
This becomes apparent in matrix factorization problems, where \texttt{ProxSPS} converges for a much wider range of regularization parameters and learning rates, while \spsmax{} 
is more sensitive to these settings. 
We also show similar results for image classification over the \texttt{CIFAR10} and \texttt{Imagenet32} dataset when using a \texttt{ResNet} model, where, compared to \texttt{AdamW}, our method is less sensitive with respect to the regularization parameter.

The remainder of our paper is organized as follows: we will first recall how the stochastic Polyak step size, in the case of $\phi=0$, can be derived using the model-based approach of \citep{Asi2019, Davis2019} and how this is connected to \spsmax{}. We then derive \texttt{ProxSPS} based on the connection to model-based methods, and present our theoretical results, based on the proof techniques in \citep{Davis2019}.
\section{Preliminaries}
\subsection{Notation}
Throughout, we will write $\E$ instead of $\E_P$. For any random variable $X(s)$, we denote $\E[X(S)] := \int_\mathcal{S} X(s) dP(s)$.
We denote $(\cdot)_+:=\max\{\cdot,0\}$. We write $\mathcal{\tilde O}$ when we drop logarithmic terms in the $\mathcal{O}$-notation, e.g. $\mathcal{\tilde O}(\frac{1}{K}) = \mathcal{O}(\frac{\ln(1+K)}{K})$. 
\subsection{General assumptions}
Throughout the article, we assume the following:
\begin{assumption}
\label{assum:inf-samples}
It is possible to generate infinitely many i.i.d.\ realizations $S_1,S_2,\dots$ from $\mathcal{S}$.
\end{assumption} 
\begin{assumption}
\label{assum:lower-bound}
For every $s\in\mathcal{S}$, $\inf_x f(x;s)$ is finite and there exists  $C(s)$ satisfying $C(s)\;\leq \;\inf_x f(x;s).$
\end{assumption} 
In many machine learning applications, non-negative loss functions are used and thus we can satisfy the second assumption choosing $C(s)=0$ for all $s\in\mathcal{S}$.
\subsection{Convex analysis}
Let $h:\R^n\to \R$ be convex and $\alpha >0$. The proximal operator is given by 
\begin{align*}
    \prox{\alpha h}(x) := \argmin_y h(y) + \frac{1}{2\alpha}\|y-x\|^2.
\end{align*}
Further, the Moreau envelope is defined by $\mathrm{env}^\alpha_h(x) := \min_y h(y) + \frac{1}{2\alpha}\|y-x\|^2$,
and its derivative is $\nabla \mathrm{env}^\alpha_h(x) = \frac{1}{\alpha}(x-\prox{\alpha h}(x))$ \citep[Lem.\ 2.1]{Drusvyatskiy2019}. Moreover, due to the optimality conditions of the proximal operator, if $h\in \mathcal{C}^1$ then 
\begin{align}
\label{eqn:env-gradient-norm}
    \hat x = \prox{\alpha h}(x) \Longrightarrow   \|\nabla h(\hat x)\| = \alpha^{-1}\|x-\hat x\| = \|\nabla \mathrm{env}^\alpha_h(x)\|.
\end{align}

\citet{Davis2019} showed how to use the Moreau envelope as a measure of stationarity: 
if $\|\nabla \mathrm{env}^\alpha_h(x)\|$ is small, then $x$ is close to $\hat{x}$ and $\hat{x}$ is an almost stationary point of $h$. Formally, the gradient of the Moreau envelope can be related to the gradient mapping (cf.\ \citep[Thm.\ 4.5]{Drusvyatskiy2019} and \cref{lem:env-gradient-mapping}).

 We say that a function $h:\R^n\to \R$ is $L$-smooth if its gradient is $L$--Lipschitz, that is
\begin{equation}\label{eq:Lsmooth}
    \|\nabla h(x) - \nabla h(y) \| \; \leq \; L \| x-y\|, \quad \forall x,y \in \R^n.
\end{equation}
If $h$ is $L$-smooth, then
\begin{align*}
    h(y) \leq h(x) + \iprod{\nabla h(x)}{y-x} +\frac{L}{2}\|y-x\|^2 \quad \text{for all } x,y, \in \R^n.
\end{align*}
A function $h:\R^n\to \R$ is $\rho$--weakly convex for $\rho\geq 0$ if $h+\tfrac{\rho}{2}\|\cdot\|^2$ is convex. Any $L$--smooth function is weakly convex with parameter less than or equal to $L$ \citep[Lem.\ 4.2]{Drusvyatskiy2019}. The above results on the proximal operator and Moreau envelope can immediately be extended to $h$ being $\rho$--weakly convex if $\alpha \in (0,\rho^{-1})$, since then $h+\tfrac{\rho}{2}\|\cdot\|^2$ is convex.
     
If we assume that each $f(\cdot;s)$ is $\rho_s$-weakly convex for $\rho_s\geq 0$, then applying \citep[Lem. 2.1]{Bertsekas1973} to the convex function $f(\cdot;s)+\tfrac{\rho_s}{2}\|\cdot\|^2$ yields that $f+ \tfrac{\rho}{2}\|\cdot\|^2$ is convex and thus $f$ is $\rho$--weakly convex for $\rho:=\E[\rho_S]$. In particular, $f$ is convex if each $f(\cdot;s)$ is assumed to be convex. For a weakly convex function $h$, we denote with $\partial h$ the regular subdifferential (cf.\ \citep[section 2.2]{Davis2019} and \citep[Def.\ 8.3]{Rockafellar1998}).
\section{The unregularized case}
\label{sec:unregularized}
For this section, consider problems of form \eqref{prob:unreg}, i.e. no regularization term $\phi$ is added to the loss $f$.
\subsection{A model-based view point}
\label{sec:model-based-view}
Many classical methods for solving \eqref{prob:unreg} can be summarized by model-based stochastic proximal point: in each iteration, a model $f_x(\cdot;s)$ is constructed approximating $f(\cdot;s)$ locally around $x$. With $S_k\sim P$ being drawn at random, this yields the update
\begin{align}
\label{eqn:model-spp-unreg}
    x^{k+1} = \arg \min_{y\in \R^n} f_{x^k}(y;S_k) + \frac{1}{2\alpha_k} \|y-x^k\|^2.
\end{align}
The theoretical foundation for this family of methods has been established by \citet{Asi2019} and \citet{Davis2019}. They give the following three models as examples:
\begin{enumerate}[label=(\roman*)]
    \item \textit{Linear:} $f_x(y;s) := f(x;s) + \langle g, y-x\rangle$ with $g\in \partial f(x;s)$.
    \item \textit{Full:} $f_x(y;s) := f(y;s)$.
    \item \textit{Truncated:} $f_x(y;s) := \max\{f(x;s) + \langle g, y-x\rangle, \inf_{z\in\R^n} f(z;s)\}$ where $g\in \partial f(x;s)$.
\end{enumerate}
It is easy to see that update \eqref{eqn:model-spp-unreg} for the \textit{linear model} is equal to \eqref{eqn:sgd} while the \textit{full model} results in the stochastic proximal point method. 
For the \textit{truncated model}, \eqref{eqn:model-spp-unreg} results in the update 
\begin{align}
\label{eqn:update-truncated}
    x^{k+1} = x^k - \min\Big\{\alpha_k, \frac{f(x^k;S_k) - \inf_{z\in\Rn} f(z;S_k)}{\|g_k\|^2}\Big\}g_k, \quad g_k \in \partial f(x^k,S_k).
\end{align}
More generally, one can replace the term $\inf_{x\in\R^n} f(x;S_k)$ with an arbitrary lower bound of $f(\cdot;S_k)$ (cf.\ \cref{lem:sps-update-unreg}). The model-based stochastic proximal point method for the truncated model is given in \cref{alg:sps-unreg}. 
The connection between the truncated model and the method depicted in \eqref{eqn:update-truncated} is not a new insight and has been pointed out in several works (including \citep{Asi2019, Loizou2021} and \citep[Prop.\ 1]{Berrada2019}). 
 For simplicity, we refer to \cref{alg:sps-unreg} as \texttt{SPS} throughout this article. However, it should be pointed out that this acronym (and variations of it) have been used for stochastic Polyak-type methods in slightly different ways \citep{Loizou2021, Gower2021}. 
%
\begin{algorithm}[H]
\caption{\texttt{SPS}}
\label{alg:sps-unreg}
\begin{algorithmic}
\REQUIRE $x^0\in\R^n$, step sizes $\alpha_k > 0$.
\FOR{$k=0,1,2,\dots,K-1$}
\STATE 1. Sample $S_k$ and set $C_k:=C(S_k)$.
\STATE 2. Choose $g_k\in \partial f(x^k;S_k)$. If $g_k=0$, set $x^{k+1}=x^k$. Otherwise, set
\begin{align}
    \label{eqn:update-sps}
    x^{k+1} = x^k - \gamma_k g_k, \quad \gamma_k = \min\Big\{\alpha_k, \frac{f(x^k;S_k) - C_k}{\|g_k\|^2}\Big\}.
\end{align}
\ENDFOR
\RETURN $x^K$
\end{algorithmic}
\end{algorithm}
%
%
%
For instance consider again the  \spsmax{} method
\begin{align}
    \tag{\spsmax}
    x^{k+1} = x^k - \min\Big\{\gamma_b, \frac{f_{i_k}(x^k)-C(s_{i_k})}{c\|\nabla f_{i_k}(x^k)\|^2}\Big\}\nabla f_{i_k}(x^k),
\end{align}
where $c, \gamma_b > 0$. Clearly, for $c=1$ and $\alpha_k=\gamma_b$, update \eqref{eqn:update-sps} is identical to \spsmax. With this in mind, we can interpret the hyperparameter $\gamma_b$ in \spsmax{} simply as a step size for the model-based stochastic proximal point step. For the parameter $c$ on the other hand, the model-based approach motivates the choice $c=1$. In this article, we will focus on this natural choice $c=1$ which also reduces the amount of hyperparameter tuning. However, we should point out that, in the strongly convex case, $c=1/2$ gives the best rate of convergence in \citep{Loizou2021}.
\section{The regularized case}
Now we consider regularized problems of the form \eqref{prob:composite}, i.e.\ 
\begin{align*}
    \min_{x \in \R^n} \psi(x), \quad \psi(x) = f(x) + \phi(x),
\end{align*}
where $\phi: \R^n \to \R \cup \{\infty\}$ is a proper, closed, $\lambda$-strongly convex function for $\lambda\geq 0$ (we allow $\lambda=0$).
For $s\in\mathcal{S}$, denote by $\psi_{x}(\cdot;s)$ a stochastic model of the objective $\psi$ at $x$. We aim to analyze algorithms with the update
\begin{align}
\label{eqn:model-spp-reg}
    x^{k+1} = \arg \min_{x\in \R^n} \psi_{x^k}(x;S_k) + \frac{1}{2\alpha_k} \|x-x^k\|^2,
\end{align}
where $S_k\sim P$ and $\alpha_k>0$. 
Naively, if we know a lower bound $\tilde{C}(s)$ of $f(\cdot;s)+\phi(\cdot)$, the truncated model could be constructed for the function $f(x;s) +\phi(x)$, resulting in 
\begin{align}
\label{eqn:model-sps}
    \psi_x(y;s) = \max\{f(x;s)+\phi(x) + \iprod{g+u}{y-x}, \tilde{C}(s) \}, \quad g\in \partial f(x;s), \quad u \in \partial \phi(x).
\end{align}
In fact, \citet{Asi2019} and \citet{Loizou2021} work in the setting of unregularized problems and hence their approaches would handle regularization in this way.
What we propose instead, is to only truncate a linearization of the loss $f(x;s)$, yielding the model
\begin{align}
\label{eqn:model-sps-prox}
    \psi_x(y;s) = f_x(y;s) + \phi(y), \quad f_x(y;s) = \max\{f(x;s)+ \iprod{g}{y-x}, C(s) \}, \quad g\in \partial f(x;s).
\end{align}
Solving \eqref{eqn:model-spp-reg} with the model in \eqref{eqn:model-sps-prox} results in 
\begin{align}
    \label{eqn:prox-sps-update-general}
    x^{k+1} = \argmin_{y\in \R^n}~~ \max\{f(x^k;S_k)+ \iprod{g_k}{y-x^k}, C(S_k) \} + \phi(y) + \frac{1}{2\alpha_k}\|y-x^k\|^2.
\end{align}
The resulting model-based stochastic proximal point method is given in \cref{alg:proxsps-general} \footnote{For $\phi=0$, \cref{alg:proxsps-general} is identical to \cref{alg:sps-unreg}.}.  \cref{lem:update-prob-general} shows that, if $\prox{\phi}$ is known, update \eqref{eqn:prox-sps-update-general} can be computed by minimizing a strongly convex function over a compact one-dimensional interval. The relation to the proximal operator of $\phi$ motivates the name \texttt{ProxSPS}.  Further,  the  \texttt{ProxSPS} update \eqref{eqn:prox-sps-update-general}  has a closed form solution when $\phi$ is the   squared $\ell_2$-norm,  as we detail in the next section.

   
%
%
\begin{algorithm}[H]
\caption{\texttt{ProxSPS}}
\label{alg:proxsps-general}
\begin{algorithmic}
\REQUIRE $x^0\in\R^n$, step sizes $\alpha_k > 0$.
\FOR{$k=0,1,2,\dots,K-1$}
\STATE 1. Sample $S_k$ and set $C_k:=C(S_k)$.
\STATE 2. Choose $g_k\in \partial f(x^k;S_k)$.\\ \hspace{2ex}
Update $x^{k+1}$ according to \eqref{eqn:prox-sps-update-general}.
\ENDFOR
\RETURN $x^K$
\end{algorithmic}
\end{algorithm}
\begin{figure}
\centering
\begin{subfigure}[t]{0.42\textwidth}
    \includegraphics[width=0.99\textwidth]{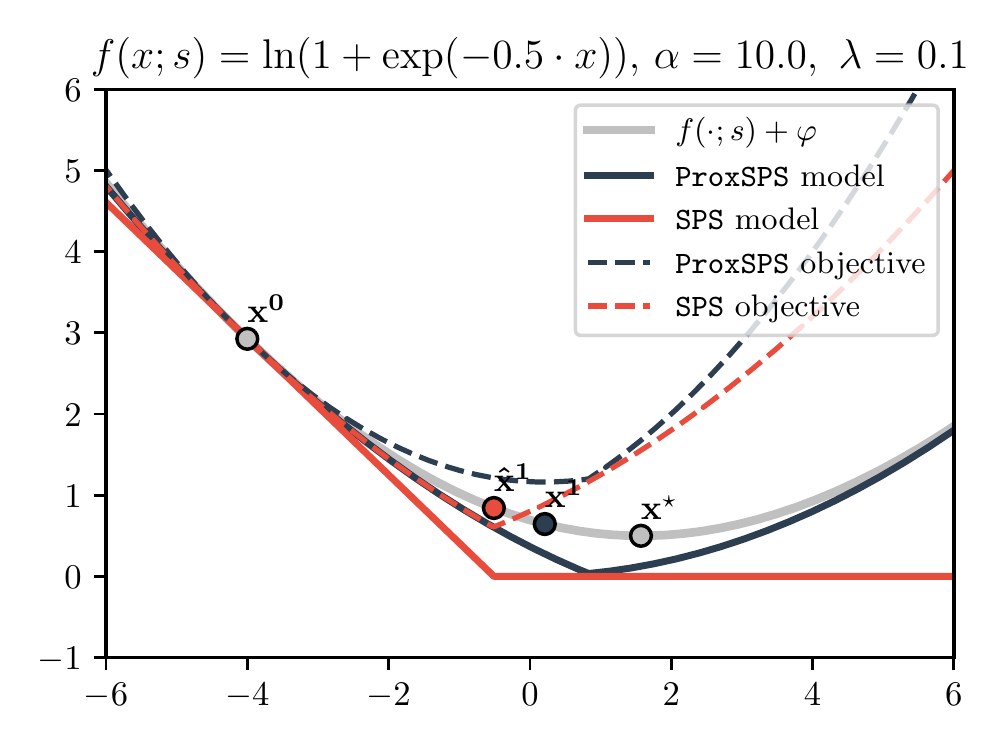}
    \caption{Regularized logistic loss. }
    \label{fig:model-compare}
\end{subfigure}
\begin{subfigure}[t]{0.52\textwidth}
    \includegraphics[width=0.99\textwidth]{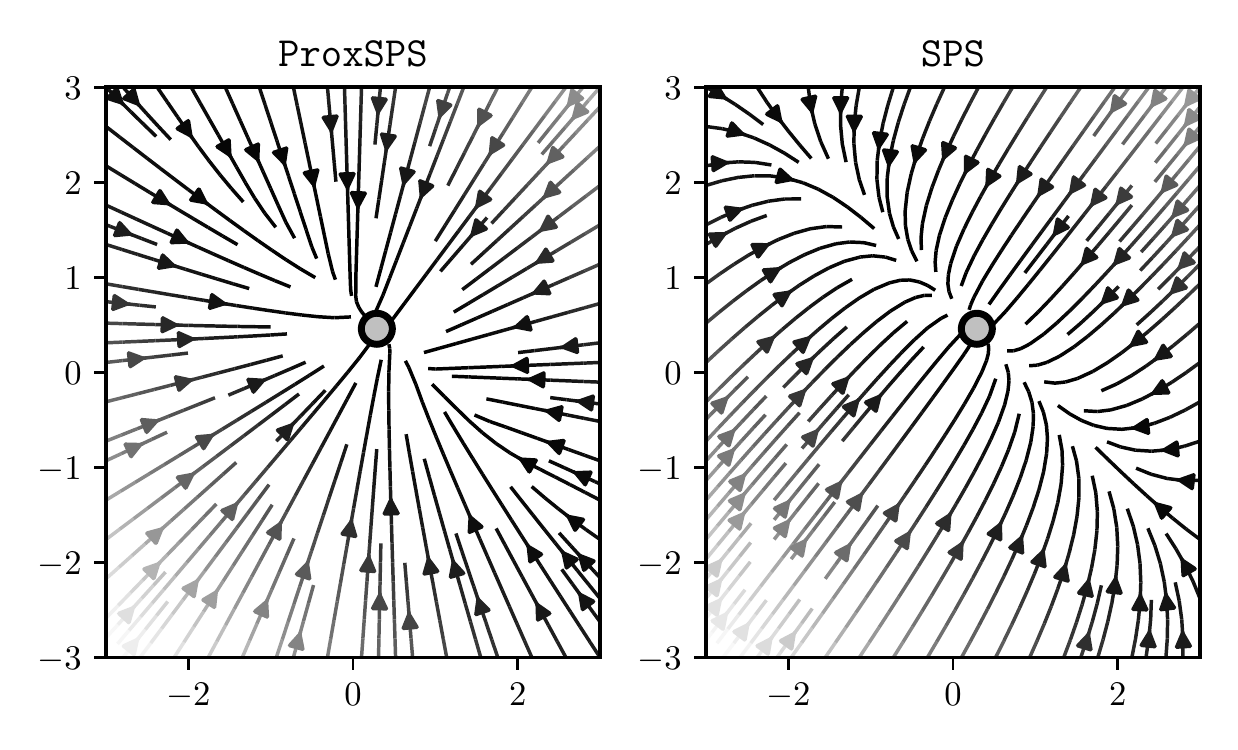}
    \caption{Regularized squared loss with $\alpha_k=1,~\lambda=1$.}
    \label{fig:model-flow}
\end{subfigure}
\caption{a) \texttt{SPS} refers to model \eqref{eqn:model-sps} whereas \texttt{ProxSPS} refers to \eqref{eqn:model-sps-prox}. We plot the corresponding model $\psi_{x^0}(y;s)$ and the objective function of \eqref{eqn:model-spp-reg}. $x^1$ (resp.\ $\hat x^1$) denotes the new iterate for \texttt{ProxSPS} (resp.\ \texttt{SPS}), $x^\star$ is the minimizer of $f(\cdot;s)+\phi$.
b) Streamlines of the vector field $V(x^k):=x^{k+1}-x^k$, for $f(x)=\|Ax-b\|^2$ and for the deterministic update, i.e.\ $f(x;s)=f(x)$. \texttt{ProxSPS} refers to update \eqref{eqn:reg-prox-sps-update} and \texttt{SPS} refers to \eqref{eqn:reg-sps-update}. The circle marks the minimizer of $f(x) + \frac{\lambda}{2}\|x\|^2$.
}
\label{fig:illustrations}
\end{figure}
%
\subsection{The special case of $\ell_2$-regularization}
\label{sec:speciall2case}
When $\phi(x) = \frac{\lambda}{2} \|x\|^2$ for some $\lambda > 0$,  \texttt{ProxSPS}~\eqref{eqn:prox-sps-update-general} has a closed form solution as we show next in~\Cref{lem:prox-sps-update}. 
For this lemma, recall that the proximal operator of $\phi(x) = \frac{\lambda}{2}\|x\|^2$ is given by $\prox{\alpha \phi}(x) = \frac{1}{1+\alpha\lambda}x$ for all  $\alpha > 0,~ x\in \R^n.$
\begin{lemma}
\label{lem:prox-sps-update}
Let $\phi(x)=\frac{\lambda}{2}\|x\|^2$ and let $g \in \partial f(x;s)$ and $C(s) \leq \inf_{z\in\R^n}f(z;s)$ hold for all $s\in\mathcal{S}$. For 
$\psi_{x}(y;s) = f_x(y;s) + \phi(y)$ with $f_x(y;s) = \max\{f(x;s) + \iprod{g}{y-x} , C(s)\}$
consider the update 
\[x^{k+1} = \argmin_{x\in \R^n} \psi_{x^k}(x;S_k) + \frac{1}{2\alpha_k} \|x-x^k\|^2.\]
Denote $C_k:=C(S_k)$ and let $g_k\in \partial f(x^k;S_k)$. Define
\begin{align*}
    \tau_k^+ := \begin{cases} 0 \quad &\text{if } g_k=0,\\ 
    \min\left\{\alpha_k, \left(\frac{(1+\alpha_k\lambda)(f(x^k;S_k) - C_k) - \alpha_k\lambda \iprod{g_k}{x^k}}{\|g_k\|^2}\right)_+\right\} \quad &\text{else.}\end{cases}
\end{align*}
Update~\eqref{eqn:prox-sps-update-general} is given by
\begin{equation}\label{eqn:prox-sps-update-l2} 
 x^{k+1}\; =\; \frac{1}{1+\alpha_k\lambda} \Big(x^k - \tau_k^+ g_k\Big) = \prox{\alpha_k \phi}(x^k - \tau_k^+ g_k). 
\end{equation} 
\end{lemma}
See~\cref{lem:prox-sps-update-long} in the appendix for an extended version of the above lemma and its proof.
The update~\eqref{eqn:prox-sps-update-l2} can be naturally decomposed into two steps,  one stochastic gradient step with an adaptive stepsize,  that is $\bar{x}^{k+1} =x^k - \tau_k^+ g_k$ followed by a proximal step $x^{k+1} = \prox{\alpha_k \phi}(\bar{x}^{k+1}).$ This decoupling into two steps, makes it easier to interpret the effect of each step,  with $\tau_k^+$ adjusting for the scale/curvature and the following proximal step shrinking the resulting parameters.  There is no clear separation of tasks if we apply the \texttt{SPS} method to the regularized problem, as we see next.

%
\begin{algorithm}[H]
\caption{\texttt{ProxSPS} for $\phi=\frac{\lambda}{2}\|\cdot\|^2$}
\label{alg:proxsps-l2}
\begin{algorithmic}
\REQUIRE $x^0\in\R^n$, step sizes $\alpha_k > 0$.
\FOR{$k=0,1,2,\dots,K-1$}
\STATE 1. Sample $S_k$ and set $C_k:=C(S_k)$.
\STATE 2. Choose $g_k\in \partial f(x^k;S_k)$. If $g_k=0$, set $x^{k+1} = \tfrac{1}{1+\alpha_k \lambda}x^k$. Otherwise, set
\begin{align*}
    x^{k+1} = \frac{1}{1+\alpha_k\lambda} \Big[x^k - \min\left\{\alpha_k, \left(\frac{(1+\alpha_k\lambda)(f(x^k;S_k) - C_k) - \alpha_k\lambda \iprod{g_k}{x^k}}{\|g_k\|^2}\right)_+\right\}g_k\Big].
\end{align*}
\ENDFOR
\RETURN $x^K$
\end{algorithmic}
\end{algorithm}
%
%
%
\subsection{Comparing the model of \texttt{SPS} and \texttt{ProxSPS} }
\label{sec:sps-max-reg-naive}
For simplicity, assume again the discrete sample space setting \eqref{eqn:ER} with differentiable loss functions $f_i$ and let $\phi=\frac{\lambda}{2}\|\cdot\|^2$. Clearly, the composite problem \eqref{prob:composite} can be transformed to an instance of \eqref{prob:unreg} by setting $\ell_i(x):= f_i(x) + \frac{\lambda}{2}\|x\|^2$ and solving
$\min_x \ell(x)$ with $\ell(x) := \frac{1}{N}\sum_{i=1}^N \ell_i(x)$.
Assume that a lower bound $\underline{\ell}_i \leq \inf_x \ell_i(x)$ is known. In this case~\eqref{eqn:model-sps} becomes
\begin{align*}
    \psi_{x}(y;s_i) = \max\Big\{f_{i}(x)+\tfrac{\lambda}{2}\|x\|^2 + \iprod{\nabla f_i(x)+ \lambda x}{y-x} ,~ \underline{\ell}_i\Big\}.
\end{align*}
Due to \cref{lem:sps-update-unreg}, if $\nabla f_{i_k}(x^k) + \lambda x^k \neq 0$, the update \eqref{eqn:model-spp-reg} is given by
\begin{align}
\label{eqn:reg-sps-update}
x^{k+1} = x^k - \min\Big\{\alpha_k, \frac{f_{i_k}(x^k)+\frac{\lambda}{2}\|x^k\|^2-\underline{\ell}_{i_k}}{\|\nabla f_{i_k}(x^k) + \lambda x^k\|^2}\Big\}(\nabla f_{i_k}(x^k) + \lambda x^k).
\end{align}
We refer to this method, which is using model \eqref{eqn:model-sps}, as \texttt{SPS}. 
On the other hand, using model \eqref{eqn:model-sps-prox} and if $\nabla f_{i_k}(x^k) \neq 0$, the update of \texttt{ProxSPS}~\eqref{eqn:prox-sps-update-l2} is
\begin{align}
    \label{eqn:reg-prox-sps-update}
    x^{k+1} = \tfrac{1}{1+\alpha_k\lambda} \Big[x^k - \min\left\{\alpha_k, \left(\tfrac{(1+\alpha_k\lambda)(f_{i_k}(x^k) - C(s_{i_k})) - \alpha_k \lambda \iprod{\nabla f_{i_k}(x^k)}{x^k}}{\|\nabla f_{i_k}(x^k)\|^2}\right)_+\right\}\nabla f_{i_k}(x^k)\Big].
\end{align}
In \cref{fig:model-compare}, we illustrate the two models \eqref{eqn:model-sps} (denoted by \texttt{SPS}) and \eqref{eqn:model-sps-prox} (denoted by \texttt{ProxSPS}) for the logistic loss with squared $\ell_2$-regularization. We can see that the \texttt{ProxSPS} model is a much better approximation of the (stochastic) objective function as it still captures the quadratic behaviour of $\phi$.  Furthermore, as noted in the previous section,  \texttt{ProxSPS} decouples the step size of the gradient and of the shrinkage,
 and hence the update direction depends on $\alpha_k$.  In contrast,  the update direction of \texttt{SPS} does not depend on $\alpha_k$,  and the regularization effect is intertwined with the adaptive step size.   
Another way to see that the model~\eqref{eqn:model-sps-prox} on which  \texttt{ProxSPS}  is based on is a more accurate model as compared to the \texttt{SPS} model~\eqref{eqn:model-sps},  is that the resulting  vector field of \texttt{ProxSPS} takes a more direct route to the minimum, as illustrated in \cref{fig:model-flow}.
 
Update \eqref{eqn:reg-prox-sps-update} needs to compute the term $\iprod{\nabla f_{i_k}(x^k)}{x^k}$ while \eqref{eqn:reg-sps-update} needs to evaluate $\|x^k\|^2$. Other than that, the computational costs are roughly identical. For \eqref{eqn:reg-prox-sps-update}, a lower bound $\underline{\ell}_i$ is required. For non-negative loss functions, in practice both $\underline{\ell}_i $ and $C(s_i)$ are often set to zero, in which case \eqref{eqn:model-sps-prox} will be a more accurate model as compared to~\eqref{eqn:model-sps}. \footnote{For single element sampling,  $\inf \ell_i$ can sometimes be precomputed (e.g.\ regularized logistic regression, see \citep[Appendix D]{Loizou2021}).  But even in this restricted setting it is not clear how to estimate $\inf \ell_i$ when using mini-batching. }

%
%
%
%
%
\subsection{Convergence analysis} \label{sec:proxsps-theory}
%
For the convergence analysis of \cref{alg:proxsps-general}, we can work with the following assumption on $\phi$.
\begin{assumption}
\label{asum:phi}
 $\phi:\R^n\to \R\cup \{\infty\}$ is a proper, closed, $\lambda$-strongly convex function with $\lambda\geq0$.
\end{assumption}
Throughout this section we consider model \eqref{eqn:model-sps-prox}, i.e.\ for $g\in \partial f(x;s)$, let
\begin{align*}
    \psi_x(y;s) = f_x(y;s) + \phi(y), \quad f_x(y;s) = \max\{f(x;s)+ \iprod{g}{y-x}, C(s) \}.
\end{align*}
Let us first state a lemma on important properties of the truncated model:
\begin{lemma}
\label{lem:sps-model-properties}
Consider $f_{x}(y;s) = \max\{f(x;s) + \iprod{g}{y-x} , C(s)\}$,
where $g \in \partial f(x;s)$ is arbitrary and $C(s) \leq \inf_{z\in\R^n}f(z;s)$. Then, it holds:
\begin{enumerate}[label=(\roman*)]
    \item \label{lem:sps-model-properties-i} The mapping $y\mapsto f_{x}(y;s)$ is convex.
    \item \label{lem:sps-model-properties-ii} For all $x\in \R^n$, it holds $f_{x}(x;s) = f(x;s)$. 
    If $f(\cdot;s)$ is $\rho_s$--weakly convex for all $s\in\mathcal{S}$, then
    \[f_{x}(y;s) \leq f(y;s)+\tfrac{\rho_s}{2}\|y-x\|^2 \quad \text{for all } x,y\in\R^n.\]
\end{enumerate}
\end{lemma}
\begin{proof}
\begin{enumerate}[label=(\roman*)]
    \item The maximum over a constant and linear term is convex.
    \item Recall that $C(s)\leq f(y;s)$ for all $y\in\R^n$. Therefore, 
    $f_{x}(x;s) = \max\{C(s), f(x;s)\} = f(x;s).$
    From weak convexity of $f(\cdot;s)$ it follows $f(x;s)+\iprod{g}{y-x}\leq f(y;s) + \tfrac{\rho_s}{2}\|y-x\|^2$ and therefore
    \[f_{x}(y;s) \leq \max\{C(s), f(y;s) + \tfrac{\rho_s}{2}\|y-x\|^2\} = f(y;s)+ \tfrac{\rho_s}{2}\|y-x\|^2 \quad \text{for all } y\in\R^n.\]
\end{enumerate}
\end{proof}
\subsubsection{Globally bounded subgradients}
In this section, we show that the results for stochastic model-based proximal point methods in \citet{Davis2019} can be immediately applied to our specific model -- even though this model has not been explicitly analyzed in their article. This, however, requires assuming that the  subgradients are bounded. 
\begin{proposition} \label{prop:one-sided-model-reg}
Let \cref{asum:phi} hold and assume that there is an open, convex set $U$ containing $\mathrm{dom}~\phi$.
Let $f(\cdot;s)$ be $\rho_s$--weakly convex for all $s\in\mathcal{S}$ and let $\rho=\E[\rho_S]$. 
Assume that there exists $G_s \in \R_+$ for all $s\in \mathcal{S}$, such that $\mathsf{G}:=\sqrt{\E[G_S^2]}< \infty$ and 
\begin{align}\label{eqn:glob-bounded-grads}
        \|g(x;s)\| \leq G_s \quad \forall g(x;s) \in \partial f(x;s),~ \forall x \in U.
\end{align}
Then, $\psi_{x}(y;s)$ satisfies \citep[Assum.\ B]{Davis2019},
in particular it holds
\begin{align}\label{eq:davis-b4-translated}
    f_{x}(x;s) - f_{x}(y;s)  \leq G_s \|x-y\| \quad \text{for all}~ s\in \mathcal{S} \text{ and all } x,y\in U.
\end{align}
\end{proposition}
\begin{remark}
    We state all four properties (B1)--(B4) of \citep[Assum.\ B]{Davis2019} explicitly in the Appendix, see \cref{prop:one-sided-model-reg-appendix} which also contains the proof. The first three properties follow immediately in our setting. Only the last property (B4), stated in \eqref{eq:davis-b4-translated}, requires the additional assumption \eqref{eqn:glob-bounded-grads}.
\end{remark}
%
%
\begin{corollary}[Weakly convex case]\label{cor:davis-weakly-convex}
Let the assumptions of \cref{prop:one-sided-model-reg} hold with $\rho_s > 0$ for all $s\in\mathcal{S}$. Let $\rho = \E[\rho_S] < \infty$ and let $\Delta \geq \env{\psi}^{1/(2\rho)}(x^0) - \min \psi$. Let $\{x^k\}_{k=0,\dots,K}$ be generated by \cref{alg:proxsps-general} for constant step sizes $\alpha_k = \Big(2\rho + \sqrt{\frac{4\rho\mathsf{G}^2K}{\Delta}}\Big)^{-1}$. Then, it holds
\begin{align*}
    \E\|\nabla \env{\psi}^{1/(2\rho)} (x^K_\sim)\|^2 \leq \frac{8\rho\Delta}{K} + 16\mathsf{G}\sqrt{\frac{\rho\Delta}{K}},
\end{align*}
where $x^K_\sim$ is uniformly drawn from $\{x^0,\dots, x^{K-1}\}$.
\end{corollary}
\begin{proof}
The claim follows from \cref{prop:one-sided-model-reg} and \citep[Thm.\ 4.3]{Davis2019}, (4.16) setting $\eta=0$, $\bar{\rho}=2\rho$, $T=K-1$ and $\beta_t = \alpha_k^{-1}$.
\end{proof}
\begin{corollary}[(Strongly) convex case]
Let the assumptions of \cref{prop:one-sided-model-reg} hold with $\rho_s = 0$ for all $s\in\mathcal{S}$. Let $\lambda > 0 $ and $x^\star = \arg \min_x \psi(x)$. Let $\{x^k\}_{k=0,\dots,K}$ be generated by \cref{alg:proxsps-general} for step sizes $\alpha_k = \tfrac{2}{\lambda (k+1)}$. Then, it holds
\begin{align*}
    \E\Big[\psi\Big(\tfrac{2}{(K+1)(K+2)-2} \sum_{k=1}^K (k+1)x^k\Big) - \psi(x^\star) \Big] \leq \frac{\lambda}{(K+1)^2}\|x^0-x^\star\|^2 + \frac{8\mathsf{G}^2}{\lambda(K+1)}.
\end{align*}
\end{corollary}
\begin{proof}
As $\rho_s = 0$ and hence $\rho=0$, we have that \citep[Assum.\ B]{Davis2019} is satisfied with $\tau=0$ (in the notation of \citep{Davis2019}, see \cref{prop:one-sided-model-reg-appendix}). Moreover, by \cref{lem:sps-model-properties}, \ref{lem:sps-model-properties-i} and $\lambda$--strong convexity of $\phi$, we have $\lambda$--strong convexity of $\psi_x(\cdot;s)$. The claim follows from \cref{prop:one-sided-model-reg} and \citep[Thm.\ 4.5]{Davis2019} setting $\mu=\lambda$, $T=K-1$ and $\beta_t = \alpha_k^{-1}$.
\end{proof}
%
\subsubsection{Lipschitz smoothness}
%
Assumption \eqref{eqn:glob-bounded-grads}, i.e.\ having globally bounded subgradients, is strong: it implies Lipschitz continuity of $f$ (cf.\ \citep[Lem.\ 4.1]{Davis2019}) and simple functions such as the squared loss do not satisfy this. Therefore, we provide additional guarantees for the smooth case, without the assumption of globally bounded gradients.

The following result, similar to \citep[Lem.\ 4.2]{Davis2019}, is the basic inequality for the subsequent convergence analysis.
\begin{lemma} \label{lem:smooth-basic-ineq}
Let \cref{asum:phi} hold. Let $x^{k+1}$ be given by \eqref{eqn:prox-sps-update-general} and $\psi_{x^k}$ be given in~\eqref{eqn:model-sps-prox}. 
For every $x\in \R^n$ it holds
\begin{align}
\label{eqn:general-basic-ineq-reg}
    (1+\alpha_k\lambda)\|x^{k+1}-x\|^2 \leq \|x^{k}-x\|^2 - \|x^{k+1}-x^k\|^2
    + 2\alpha_k\big(\psi_{x^k}(x;S_k) - \psi_{x^k}(x^{k+1};S_k)\big).
\end{align}
Moreover, it holds
\begin{align}
    \label{eqn:general-model-value-reg}
    \psi_{x^k}(x^{k+1};S_k) \geq f(x^k;S_k) + \iprod{g_k}{x^{k+1}-x^k} + \phi(x^{k+1}).
\end{align}
\end{lemma}
\begin{proof}
The objective of \eqref{eqn:prox-sps-update-general} is given by
$\Psi_k(y):= \psi_{x^k}(y;S_k) + \frac{1}{2\alpha_k} \|y-x^k\|^2.$
Using \cref{lem:sps-model-properties}, \ref{lem:sps-model-properties-i} and $\lambda$-strong convexity of $\phi$, $\Psi_k(y)$ is $(\lambda+\frac{1}{\alpha_k})$--strongly convex. As $x^{k+1}$ is the minimizer of $\Psi_k(y)$, for all $x\in\R^n$ we have
\begin{align*}
    &\Psi_k(x) \geq \Psi_k(x^{k+1}) + \frac{1+\alpha_k\lambda}{2\alpha_k}\|x^{k+1}-x\|^2  \iff \nonumber \\
    & (1+\alpha_k\lambda)\|x^{k+1}-x\|^2 \leq \|x^{k}-x\|^2 - \|x^{k+1}-x^k\|^2
    + 2\alpha_k\big(\psi_{x^k}(x;S_k) - \psi_{x^k}(x^{k+1};S_k)\big).
\end{align*}
Moreover, by definition of $f_x(y;s)$ in \eqref{eqn:model-sps-prox} we have
\begin{align*}
    \psi_{x^k}(x^{k+1};S_k) = f_{x^k}(x^{k+1};S_k) + \phi(x^{k+1}) \geq f(x^k;S_k) + \iprod{g_k}{x^{k+1}-x^k} + \phi(x^{k+1}).
\end{align*}
\end{proof}
We will work in the setting of differentiable loss functions with bounded gradient noise.
\begin{assumption} \label{asum:gradient-noise}
The mapping $f(\cdot;s)$ is differentiable for all $s\in \mathcal{S}$ and there exists $\beta \geq 0$ such that
\begin{align}
    \label{eqn:asum-noise-reg}
    \E\|\nabla f(x;S) - \nabla f(x)\|^2 \leq  \beta \quad \text{for all } x\in\R^n.
\end{align}
\end{assumption}
The assumption of bounded gradient noise \eqref{eqn:asum-noise-reg} (in the differentiable setting) is indeed a weaker assumption than \eqref{eqn:glob-bounded-grads} since $\E[\nabla f(x;S)]=\nabla f(x)$ and
\[\E\|\nabla f(x;S) - \nabla f(x)\|^2 \leq  \beta \iff \E\|\nabla f(x;S)\|^2 \leq \|\nabla f(x)\|^2 + \beta.\]
\begin{remark}
\cref{asum:gradient-noise} (and the subsequent theorems) could be adapted to the case where $f(\cdot;s)$ is weakly convex but non-differentiable: for fixed $x\in\R^n$, due to \citep[Prop.\ 2.2]{Bertsekas1973} and \citep[Lem.\ 2.1]{Davis2019} it holds
\begin{align*}
    \E[\partial f(x;S)] = \E\Big[\partial \big(f(x;S)+ \frac{\rho_S}{2}\|x\|^2\big) - \rho_S x\Big] = \partial f(x) + \rho x - \E[\rho_S x] = \partial f(x),
\end{align*}
where we used $\rho=\E[\rho_S]$. Hence, for $g_s \in \partial f(x;s)$ we have $\E[g_S] \in \partial f(x)$ and \eqref{eqn:asum-noise-reg} is replaced by 
\begin{align*}
    \E\|g_S - \E[g_S]\|^2 \leq \beta \quad \text{for all } x\in\R^n.
\end{align*}
However, as we will still require that $f$ is Lipschitz-smooth, we present our results for the differentiable setting.
\end{remark}
The proof of the subsequent theorems can be found in \cref{sec:proof-thm-conv} and \cref{sec:proof-thm-nonconv}.
\begin{theorem}
\label{thm:convex-smooth-reg}
Let \cref{asum:phi} and \cref{asum:gradient-noise} hold.
Let $f(\cdot;s)$ be convex for all $s\in\mathcal{S}$ and let $f$ be $L$--smooth~\eqref{eq:Lsmooth}.  Let $x^\star = \argmin_{x\in\R^n} \psi(x)$ and let $\theta>1$. Let $\{x^k\}_{k=0,\dots,K}$ be generated by \cref{alg:proxsps-general} for step sizes $\alpha_k > 0$ such that
\begin{align}
\label{eqn:alpha-cond-reg}
   \alpha_k \leq \frac{1-1/\theta}{L}.
\end{align}
Then, it holds
\begin{align}\label{eqn:main-recursion-reg}
    (1+\alpha_k\lambda)\E\|x^{k+1}-x^\star\|^2 \leq \E\|x^{k}-x^\star\|^2 + 2\alpha_k\E[\psi(x^\star)-\psi(x^{k+1})] + \theta\beta\alpha_k^2.
\end{align}
Moreover, we have:
\begin{enumerate}[label=\alph*)]
    \item \label{item-a:convex-smooth-reg} 
    If $\lambda > 0 $ and $\alpha_k = \tfrac{1}{\lambda(k+k_0)}$ with $k_0 \geq 1$ large enough such that \eqref{eqn:alpha-cond-reg} is fulfilled, then
            \begin{align} \label{eqn:strong-conv-decay-reg}
                \E\Big[\psi\Big(\tfrac{1}{K}\sum_{k=0}^{K-1}x^{k+1} \Big) -\psi(x^\star)\Big] \leq \frac{\lambda k_0}{2K}\|x^0-x^\star\|^2 + \frac{\theta\beta (1+\ln K)}{2\lambda K}.
            \end{align}
    \item \label{item-b:convex-smooth-reg}
    If $\lambda = 0$ and $\alpha_k = \tfrac{\alpha}{\sqrt{k+1}}$ with $\alpha \leq \frac{1-1/\theta}{L}$, then 
        \begin{align}\label{eqn:conv-decay-reg}
            \E\Big[\psi\Big(\tfrac{1}{\sum_{k=0}^{K-1}\alpha_k}\sum_{k=0}^{K-1}\alpha_k x^{k+1} \Big) -\psi(x^\star)\Big] \leq \frac{\|x^0-x^\star\|^2}{4\alpha(\sqrt{K+1}-1)} + \frac{\theta\beta\alpha (1+\ln K)}{4(\sqrt{K+1} -1)}.
        \end{align}
    \item \label{item-c:convex-smooth-reg} 
    If $f$ is $\mu$--strongly convex with $\mu\geq 0$,\footnote{Note that as $f(\cdot;s)$ is convex, so is $f$, and that we allow $\mu=0$ here.} and $\alpha_k = \alpha$ fulfilling \eqref{eqn:alpha-cond-reg}, then
            \begin{align}
                \label{eqn:strong-conv-cst-reg}
                \E\|x^{K}-x^\star\|^2 \leq (1+\alpha(\mu+2\lambda))^{-K} \|x^0-x^\star\|^2 + \frac{\theta\beta\alpha}{\mu+2\lambda}.
            \end{align}
\end{enumerate}
\end{theorem}
\begin{remark}
If $\lambda>0$, for the decaying step sizes in item \ref{item-a:convex-smooth-reg} we get a rate of $\mathcal{\tilde{O}}(\tfrac{1}{K})$ if $\lambda > 0$. In the strongly convex case in item \ref{item-c:convex-smooth-reg}, for constant step sizes, we get a linear convergence upto a neighborhood of the solution. Note that the constant on the right-hand side of \eqref{eqn:strong-conv-cst-reg} can be forced to be small using a small $\alpha$. Further, the rate~\eqref{eqn:strong-conv-cst-reg} has a
 $2\lambda$ term, instead of $\lambda$. This slight improvement in the rate occurs because we do not linearize $\phi$ in the \texttt{ProxSPS} model.
\end{remark}
%
\begin{theorem}
\label{thm:exact-nonconv-reg}
Let \cref{asum:phi} and \cref{asum:gradient-noise} hold.
Let $f(\cdot;s)$ be $\rho_s$--weakly convex for all $s\in\mathcal{S}$ and let $\rho :=\E[\rho_S] < \infty$. Let $f$ be $L$--smooth\footnote{As $f$ is $\rho$--weakly convex, this implies $\rho\leq L$.} and assume that $\inf \psi > - \infty$. Let $\{x^k\}_{k\geq 0}$ be generated by \cref{alg:proxsps-general}.
For $\theta > 1$, under the condition
\begin{align}
    \label{eqn:cond-eta-alpha-reg}
    \eta \in \begin{cases} (0, \tfrac{1}{\rho-\lambda}) &\text{if } \rho > \lambda\\ (0,\infty)&\text{else}\end{cases}, \hspace{15ex} \alpha_k \leq \frac{1-\theta^{-1}}{L+\eta^{-1}},
\end{align}
it holds
\begin{align}\label{eqn:moreau-env-recurse-reg}
    \sum_{k=0}^{K-1}\alpha_k \E \|\nabla \mathrm{env}^\eta_\psi(x^k)\|^2 \leq \frac{2(\mathrm{env}^\eta_\psi(x^0) -\inf \psi)}{1-\eta(\rho-\lambda)}+ \frac{\beta\theta}{\eta(1-\eta(\rho-\lambda))}\sum_{k=0}^{K-1}\alpha_k^2.
\end{align}
Moreover, for the choice $\alpha_k = \frac{\alpha}{\sqrt{k+1}}$ and with $\alpha \leq \frac{1-\theta^{-1}}{L+\eta^{-1}}$, we have
\begin{align*}
    \min_{k=0,\dots,K-1}\E \|\nabla \mathrm{env}^\eta_\psi(x^k)\|^2 \leq \frac{\mathrm{env}^\eta_\psi(x^0) -\inf \psi}{\alpha(1-\eta(\rho-\lambda))(\sqrt{K+1}-1)} + \frac{\beta\theta}{2\eta(1-\eta(\rho-\lambda))}\frac{\alpha(1+\ln K)}{(\sqrt{K+1}-1)}.
\end{align*}
If instead we choose $\alpha_k = \frac{\alpha}{\sqrt{K}}$ and with $\alpha \leq \sqrt{K}\frac{1-\theta^{-1}}{L+\eta^{-1}}$, we have
\begin{align*}
    \E \|\nabla \mathrm{env}^\eta_\psi(x^K_\sim)\|^2 \leq \frac{2(\mathrm{env}^\eta_\psi(x^0) -\inf \psi)}{\alpha(1-\eta(\rho-\lambda))\sqrt{K}} + \frac{\beta\theta}{\eta(1-\eta(\rho-\lambda))}\frac{\alpha}{\sqrt{K}},
\end{align*}
where $x^K_\sim$ is uniformly drawn from $\{x^0,\dots, x^{K-1}\}$.
\end{theorem}
\subsubsection{Comparison to existing theory}

Recalling that \cref{alg:sps-unreg} is equivalent to \spsmax{} with $c=1$ and $\gamma_b=\alpha_k$, we can apply \cref{thm:convex-smooth-reg} and \cref{thm:exact-nonconv-reg} for the unregularized case where $\phi=0$ and hence obtain new theory for \eqref{eqn:sps-max}. 
We start by summarizing the main theoretical results for \spsmax{} given in \citep{Loizou2021,Orvieto2022}: in the \eqref{eqn:ER} setting, recall the interpolation constant $\sigma^2 = \E[f(x^\star;S)-C(S)] = \frac{1}{N}\sum_{i=1}^N f_i(x^\star) - C(s_i)$. If $f_i$ is $L_i$-smooth and convex, \citep[Thm.\ 3.1]{Orvieto2022} proves convergence to a neighborhood of the solution, i.e. the iterates $\{x^k\}$ of \spsmax{} satisfy
\begin{align}\label{eq:spsconvex}
    \E[f(\bar{x}^K) - f(x^\star)] \leq \frac{\|x^0-x^\star\|^2}{\alpha K} + \frac{2\gamma_b\sigma^2}{\alpha},
\end{align}
where $\bar{x}^K := \frac{1}{K}\sum_{k=0}^{K-1} x^k$, $\alpha:= \min\{\frac{1}{2cL_{\max}}, \gamma_b\}$, and $L_{\max} := \max_{i\in [N]} L_i$.\footnote{The theorem also handles the mini-batch case but, for simplicity, we state the result for sampling a single $i_k$ in each iteration.}
For the nonconvex case, if $f_i$ is $L_i$-smooth and under suitable assumptions on the gradient noise, \citep[Thm.\ 3.8]{Loizou2021} states that, for constants $c_1$ and $c_2$, we have 
\begin{align} \label{eq:spsnonconvex}
    \min_{k=1,\dots,K} \E\|\nabla f(x^k)\|^2  \leq \frac{1}{c_1 K} + c_2.
\end{align}
The main advantage of these results is that $\gamma_b$ can be held constant; furthermore in the convex setting~\eqref{eq:spsconvex}, the choice of $\gamma_b$ requires no knowledge of the smoothness constants $L_i$.
For both results however, we can not directly conclude that the right-hand side goes to zero as $K\to \infty$ as there is an additional constant. Choosing $\gamma_b$ sufficiently small does not immediately solve this as $c_1$, $\alpha$ and $c_2$ all go to zero as $\gamma_b$ goes to zero. \\
Our results complement this by showing exact convergence for the (weakly) convex case, i.e.\ without constants on the right-hand side. This comes at the cost of an upper bound on the step sizes $\alpha_k$ which depends on the smoothness constant $L$. For exact convergence, it is important to use decreasing step sizes $\alpha_k$: \cref{thm:exact-nonconv-reg} shows that the gradient of the Moreau envelope converges to zero at the rate $\mathcal{O}(1/\sqrt{K})$ for the choice of $\alpha_k=\frac{\alpha}{\sqrt{K}}$.\footnote{Notice that $\alpha_k$ then depends on the total number of iterations $K$ and hence one would need to fix $K$ before starting the method.} 
Another minor difference to \citep{Loizou2021} is that we do not need to assume Lipschitz-smoothness for all $f(\cdot;s)$ and work instead with the (more general) assumption of weak convexity. However, we still need to assume Lipschitz smoothness of $f$. 

Another variant of \spsmax{}, named \texttt{DecSPS}, has been proposed in \citep{Orvieto2022}: for unregularized problems \eqref{prob:unreg} it is given by
\begin{align}\label{alg:decsps}
\tag{\texttt{DecSPS}}
    x^{k+1} = x^k - \hat\gamma_k g_k, \quad \hat \gamma_k = \frac{1}{c_k}\min\Big\{\frac{f(x^k;S_k)-C_k}{\|g_k\|^2}, c_{k-1}\hat \gamma_{k-1}\Big\}
\end{align}
where $\{c_k\}_{k\geq 0}$ is an increasing sequence. In the \eqref{eqn:ER} setting, if all $f_i$ are Lipschitz-smooth and strongly convex, \texttt{DecSPS} converges with a rate of $\mathcal{O}(\frac{1}{\sqrt{K}})$, without knowledge of the smoothness or convexity constants (cf.\ \citep[Thm.\ 5.5]{Orvieto2022}). However, under these assumptions, the objective $f$ is strongly convex and the optimal rate is $\mathcal{O}(\frac{1}{K})$, which we achieve up to a logarithmic factor in \cref{thm:convex-smooth-reg}, \eqref{eqn:strong-conv-decay-reg}. Moreover, for \texttt{DecSPS} no guarantees are given for nonconvex problems.

For regularized problems, the constant in \eqref{eq:spsconvex} is problematic if $\sigma^2$ (computed for the regularized loss) is moderately large. We refer to \cref{sec:interpolation-constant} where we show that this can easily happen. For \texttt{ProxSPS}, our theoretical results \cref{thm:convex-smooth-reg} and \cref{thm:exact-nonconv-reg} are not affected by this as they do not depend on the size of $\sigma^2$. To the best of our knowledge, this is the first work to show theory for the stochastic Polyak step size in a setting that explicitly considers regularization. Moreover, our results also cover the case of non-smooth or non-real-valued regularization $\phi$ where the theory in \citep{Loizou2021} can not be applied.

%
%
%
\section{Numerical experiments}
Throughout we denote \cref{alg:sps-unreg} with \texttt{SPS} and \cref{alg:proxsps-l2} with \texttt{ProxSPS}. 
For all experiments we use \texttt{PyTorch} \citep{Paszke2019}\footnote{The code for our experiments and an implementation of \texttt{ProxSPS} can be found at \url{https://github.com/fabian-sp/ProxSPS}.}.
\subsection{General parameter setting}
\label{sec:numerics-general}
For \texttt{SPS} and \texttt{ProxSPS} we always use $C(s)=0$ for all $s\in \mathcal{S}$. For $\alpha_k$, we use the following schedules:
\begin{itemize}
    \item \texttt{constant}: set $\alpha_k = \alpha_0$ for all $k$ and some $\alpha_0>0$.
    \item \texttt{sqrt}: set $\alpha_k = \tfrac{\alpha_0}{\sqrt{j}}$ for all iterations $k$ during epoch $j$.
\end{itemize}
As we consider problems with $\ell_2$-regularization, for \texttt{SPS} we handle the regularization term by incorporating it into all individual loss functions, as depicted in \eqref{eqn:reg-sps-update}. 
With $\phi=\frac{\lambda}{2} \|\cdot\|^2$ for $\lambda\geq 0$, we denote by $\zeta_k$ the \emph{adaptive step size} term of the following algorithms:
\begin{itemize}\label{it:adaptivestep}
    \item for \texttt{SPS} we have  $\zeta_k := \frac{f(x^k;S_k) + \frac{\lambda}{2}\|x^k\|^2}{\|g_k + \lambda x^k\|^2}$ (cf.\ \eqref{eqn:reg-sps-update} with $\underline{\ell}_{i_k}=0$ ),
    \item for \texttt{ProxSPS} we have $\zeta_k := \left(\frac{(1+\alpha_k\lambda)f(x^k;S_k)  - \alpha_k\lambda \iprod{g_k}{x^k}}{\|g_k\|^2}\right)_+$ and thus $\tau_k^+ = \min\{\alpha_k, \zeta_k\}$ (cf.\ \cref{lem:prox-sps-update} with $C(S_k)=0$).
\end{itemize}
%
\subsection{Regularized matrix factorization}
\textbf{Problem description:} For $A\in\R^{q\times p}$, consider the problem
\begin{align*}
    \min_{W_1 \in \R^{r\times p}, W_2\R^{q\times r}} \E_{y\sim N(0,I)} \|W_2W_1y - Ay\|^2 = \min_{W_1 \in \R^{r\times p}, W_2\R^{q\times r}} \|W_2W_1-A\|_F^2.
\end{align*}
For the above problem, \spsmax{} has shown superior performance than other methods in the numerical experiments of \citep{Loizou2021}. The problem can can be turned into a (nonconvex) empirical risk minimization problem by drawing $N$ samples $\{y^{(1)},\dots,y^{(N)}\}$. Denote $b^{(i)}:=Ay^{(i)}$. Adding squared norm regularization with $\lambda\geq 0$ (cf.\ \citep{Srebro2004}), we obtain the problem
\begin{align}
\label{prob-matrix-fac-reg}
    \min_{W_1 \in \R^{r\times p}, W_2\R^{q\times r}} \oneover{N} \sum_{i=1}^N \|W_2 W_1 y^{(i)} - b^{(i)}\|^2 + \tfrac{\lambda}{2}\big(\|W_1\|_F^2 + \|W_2\|_F^2\big).
\end{align}
This fits the format of \eqref{prob:composite}, where $x=(W_1,W_2)$, using a finite sample space $\mathcal{S}=\{s_1,\dots,s_N\}$, $f(x;s_i) = \|W_2W_1 y^{(i)} - Ay^{(i)}\|^2$, and $\phi = \tfrac{\lambda}{2}\|\cdot\|_F^2$. Clearly, zero is a lower bound of $f(\cdot;s_i)$ for all $i\in[N]$.
We investigate \texttt{ProxSPS} for problems of form \eqref{prob-matrix-fac-reg} on synthetic data. For details on the experimental procedure, we refer to~\cref{sec:appendix-numerics-matrixfac}.
\vspace{2mm}\\
\textbf{Discussion:} We discuss the results for the setting \texttt{matrix-fac1} in \cref{table:matrix-fac} in the Appendix. We first fix $\lambda=0.001$ and consider the three methods \texttt{SPS}, \texttt{ProxSPS} and \texttt{SGD}. \cref{fig:matrix_fac1_stability} shows the objective function over 50 epochs, for both step size schedules \texttt{sqrt} and \texttt{constant}, and several initial values $\alpha_0$. For the \texttt{constant} schedule, we observe that \texttt{ProxSPS} converges quickly for all initial values while \texttt{SPS} is unstable. Note that for \texttt{SGD} we need to pick much smaller values for $\alpha_0$ in order to avoid divergence 
(\texttt{SGD} diverges for large $\alpha_0$).  \texttt{SPS} for large $\alpha_0$ is unstable, while for small $\alpha_0$ we can expect similar performance to \texttt{SGD} (as $\gamma_k$ is capped by $\alpha_k=\alpha_0$). However, in the regime of small $\alpha_0$, convergence will be very slow. Hence, one of the main advantages of \texttt{SPS}, namely that its step size can be chosen constant and moderately large (compared to \texttt{SGD}), is not observed here. \texttt{ProxSPS} fixes this by admitting  a larger range of initial step sizes, all of which result in fast convergence,
and therefore is more robust than \texttt{SGD} and \texttt{SPS} with respect to the tuning of $\alpha_0$.
   
For the \texttt{sqrt} schedule, we  observe in \cref{fig:matrix_fac1_stability} that \texttt{SPS} can be stabilized by reducing the values of $\alpha_k$ over the course of the iterations. However, for large $\alpha_0$ we still see instability in the early iterations, whereas \texttt{ProxSPS} does not show this behaviour. We again observe that \texttt{ProxSPS} is less sensitive with respect to the choice of $\alpha_0$ as compared to \texttt{SGD}. 
The empirical results also confirm our theoretical statement, showing exact convergence if $\alpha_k$ is decaying in the order of $1/\sqrt{k}$. 
From \cref{fig:matrix_fac1_stability_val}, we can make similar observations for the validation error, defined as
$\oneover{N_\text{val}} \sum_{i=1}^{N_\text{val}} \|W_2 W_1 y^{(i)} - b^{(i)}_\text{val}\|^2$,
where $b^{(i)}_\text{val}$ are the $N_\text{val} = N$ measurements from the validation set (cf.\ \cref{sec:appendix-numerics-matrixfac} for details).
\begin{figure}[t]
    \centering
    \includegraphics[height=0.25\textheight]{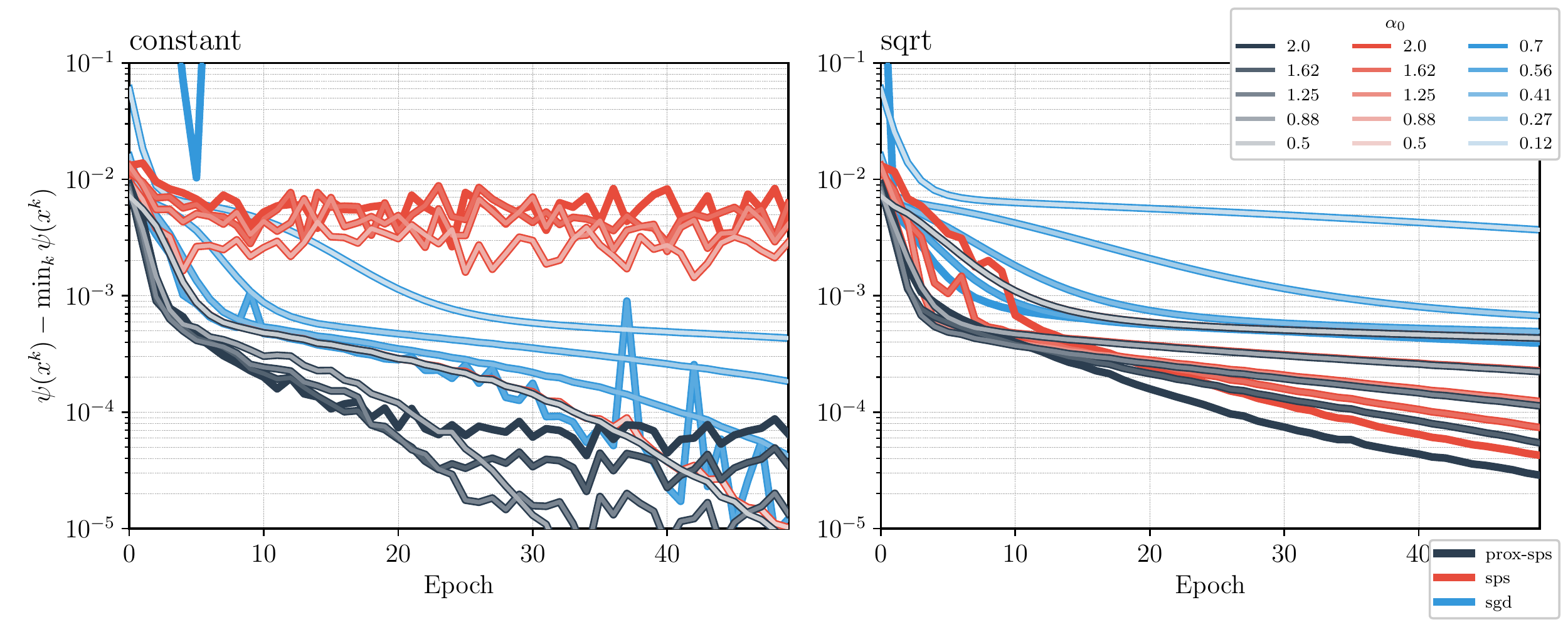}
    \caption{ Objective function for  the Matrix Factorization problem~\eqref{prob-matrix-fac-reg}, with \texttt{constant} (left) and \texttt{sqrt} (right) step size schedule and several choices of initial values. Here $\min_k \psi(x^k)$ is the best objective function value found over all methods and all iterations.}
    \label{fig:matrix_fac1_stability}
\end{figure}
\begin{figure}[t]
    \centering
    \includegraphics[height=0.25\textheight]{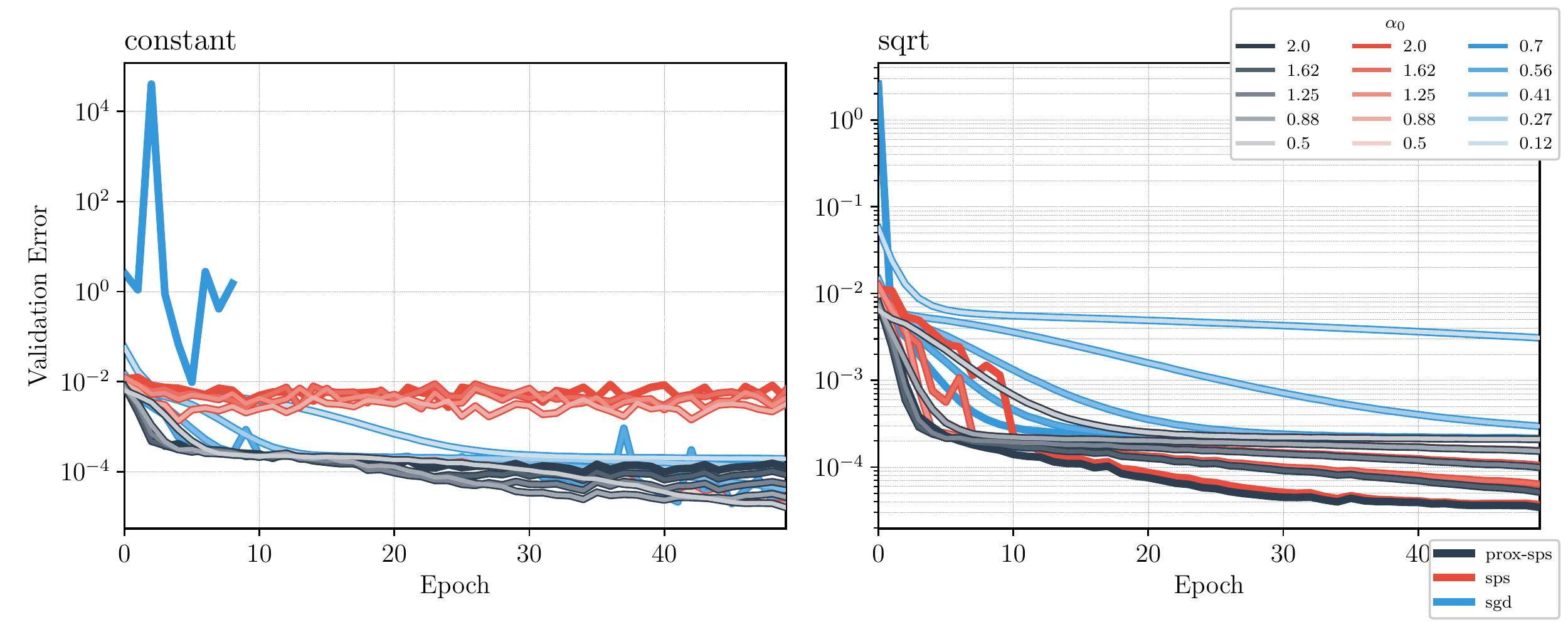}
    \caption{ Validation error for   the Matrix Factorization problem~\eqref{prob-matrix-fac-reg}, with \texttt{constant} (left) and \texttt{sqrt} (right) step size schedule and several choices of initial values. }
    \label{fig:matrix_fac1_stability_val}
\end{figure}

  We now consider different values for $\lambda$ and only consider the \texttt{sqrt} schedule, as we have seen that for constant step sizes, \texttt{SPS} would not work for large step sizes and be almost identical to \texttt{SGD} for small step sizes. \cref{fig:matrix_fac1} shows the objective function and validation error. Again, we can observe that \texttt{SPS} is unstable for large initial values $\alpha_0$ for all $\lambda\geq 10^{-4}$. On the other hand, \texttt{ProxSPS} has a good performance for a wide range of $\alpha_0 \in[1,10]$ if $\lambda$ is not too large.  Indeed,  \texttt{ProxSPS} convergence only starts to deteriorate  when both $\alpha_0$ and $\lambda$ are very large. For $\alpha_0=1$, the two methods give almost identical results.
\begin{figure}[t]
    \centering
    \includegraphics[width=0.99\textwidth]{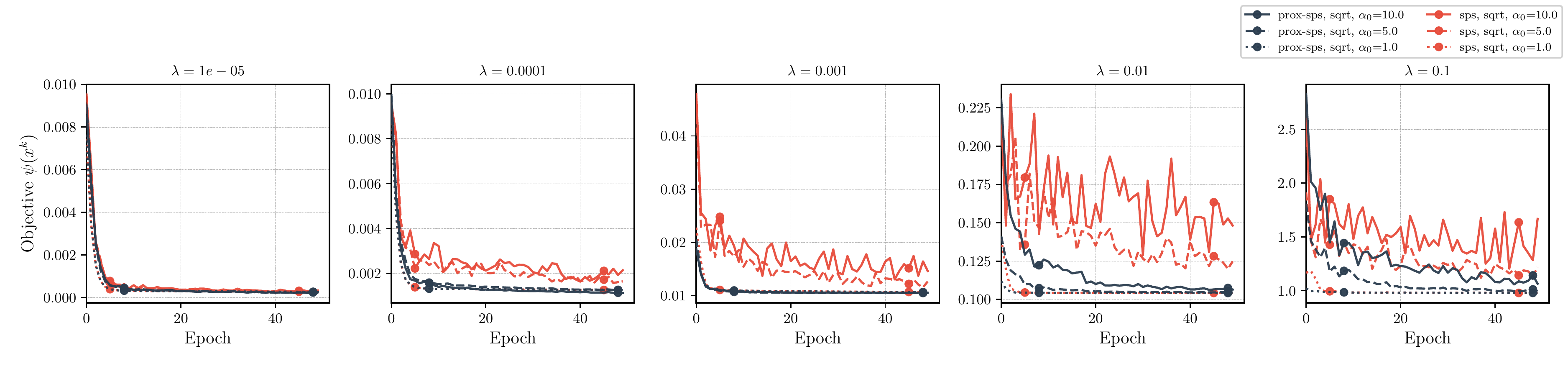}
    \includegraphics[width=0.99\textwidth]{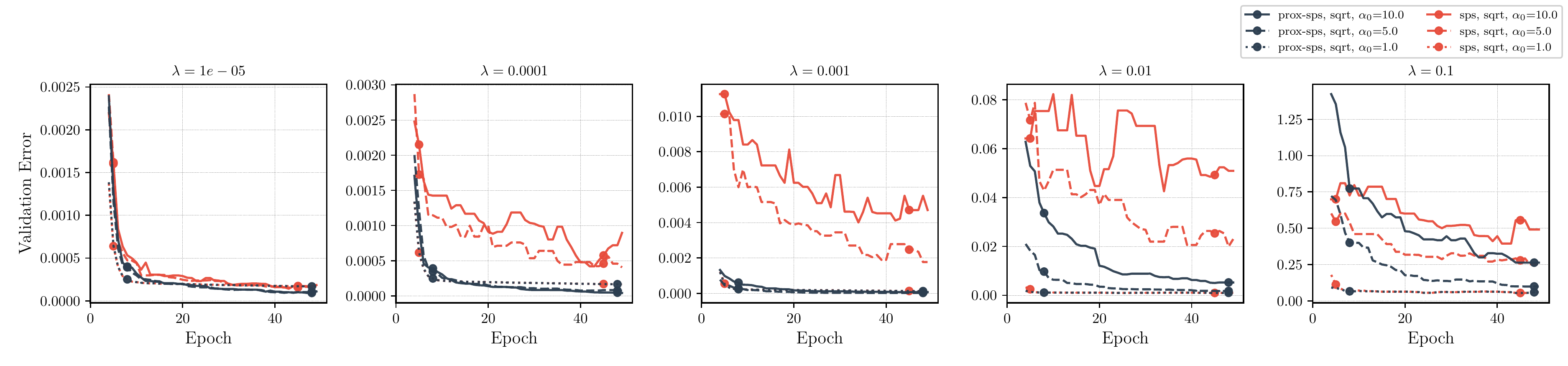}
    \caption{ Objective function value and validation error over the course of optimization. For the validation error, we plot a rolling median over five epochs in order to avoid clutter.}
    \label{fig:matrix_fac1}
\end{figure}
Finally, in \cref{fig:matrix_fac1_lambda_path_val} we plot the validation error as a function of $\lambda$ (taking the median over the last ten epochs). The plot shows that the best validation error is obtained for $\lambda=10^{-4}$ and for large $\alpha_0$. 
With \texttt{SPS} the validation error is higher, in particular for large $\alpha_0$ and $\lambda$. \cref{fig:matrix_fac1_lambda_path_norm} shows that \texttt{ProxSPS} leads to smaller norm of the iterates, hence a more effective regularization.
\begin{figure}[t]
    \centering
    \begin{subfigure}[b]{0.48\textwidth}
        \centering
        \includegraphics[height=0.2\textheight]{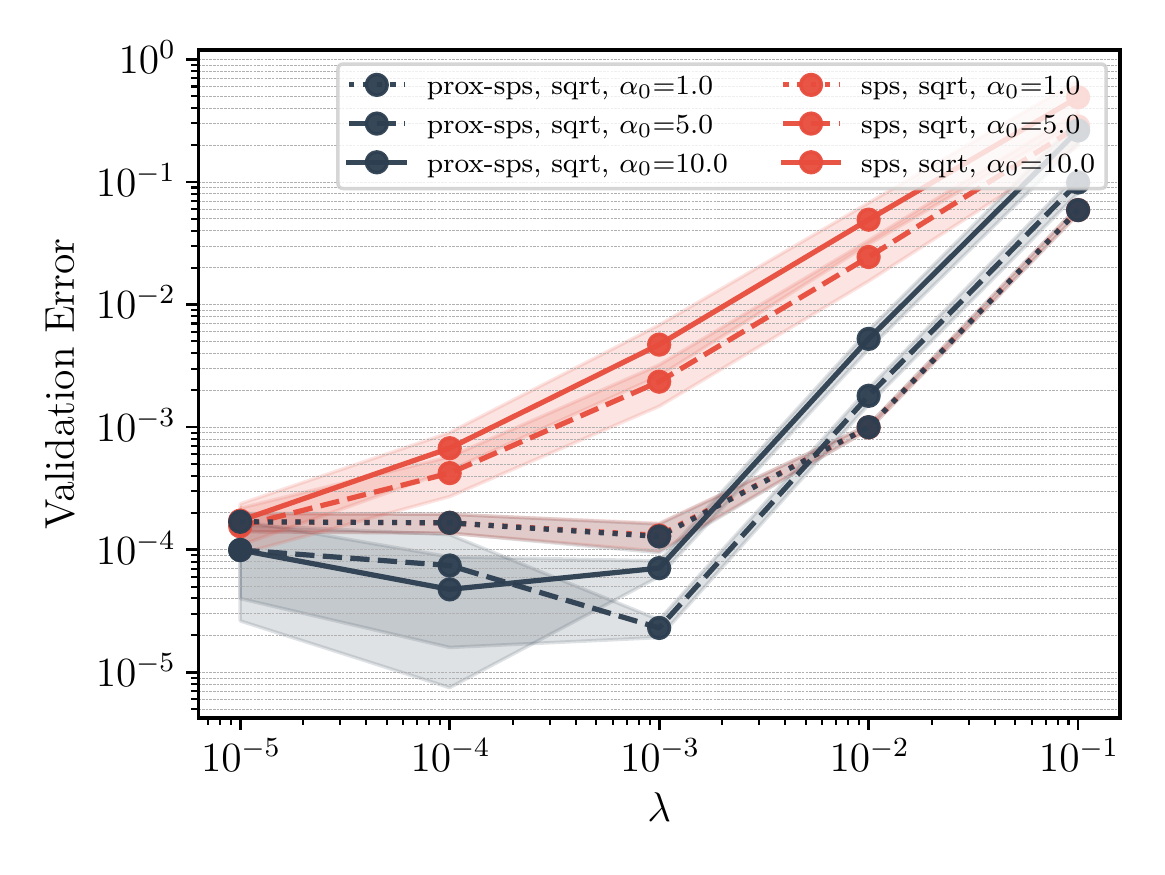}
        \caption[]%
        {{\small Validation error}}    
        \label{fig:matrix_fac1_lambda_path_val}
    \end{subfigure}
    \begin{subfigure}[b]{0.48\textwidth}  
        \centering 
        \includegraphics[height=0.2\textheight]{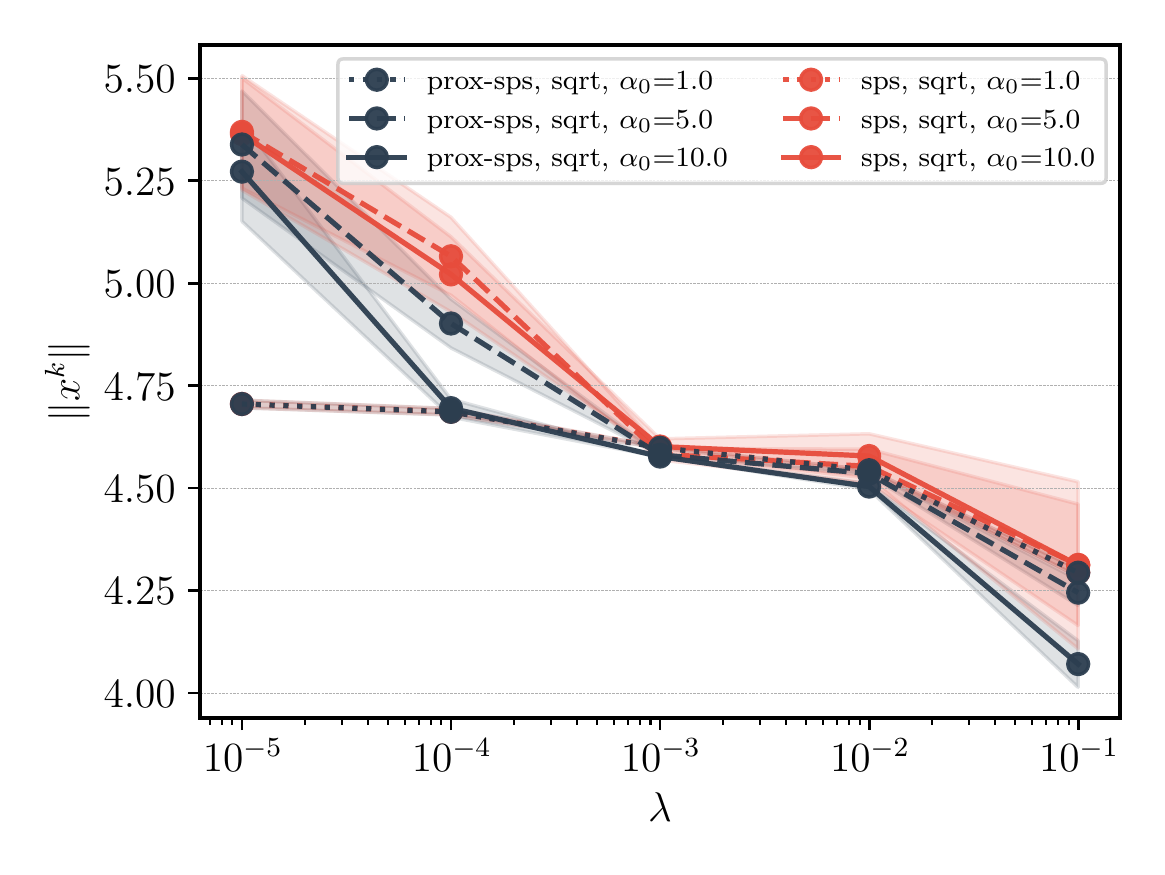}
        \caption[]%
        {{\small Model norm $\sqrt{\|W_1\|^2 + \|W_2\|^2}$}}    
        \label{fig:matrix_fac1_lambda_path_norm}
    \end{subfigure}
\caption[]
{Validation error and model norm as a function of the regularization parameter $\lambda$. Shaded area is one standard deviation (computed over ten independent runs). For all values, we take the median over epochs $[40,50]$.}
\label{fig:matrix_fac1_lambda_path}
\end{figure}
Finally, we plot the actual step sizes for both methods in \cref{fig:matrix_fac1_step_sizes}. We observe that the adaptive step size $\zeta_k$ (Definition at end of \cref{it:adaptivestep}) is typically larger and has more variance for \texttt{SPS} than \texttt{ProxSPS}, in particular for large $\lambda$. This increased variance might explain why \texttt{SPS} is unstable when $\alpha_0$ is large: the actual step size is the minimum between $\alpha_k$ and $\zeta_k$ and hence both terms being large could lead to instability. On the other hand, if $\alpha_0=1$, the plot confirms that \texttt{SPS} and \texttt{ProxSPS} are almost identical methods as $\zeta_k > \alpha_k$ for most iterations. 
%

We provide additional numerical results which confirm the above findings in the Appendix: this includes the results for the setting \texttt{matrix-fac2} of \cref{table:matrix-fac} in \cref{sec:matrix-fac2-plots} as well as a matrix completion task on a real-world dataset of air quality sensor networks \citep{RiveraMunoz2022} in \cref{sec:exp-matrix-completion}.

\begin{figure}[t]
    \centering
    \includegraphics[width=0.99\textwidth]{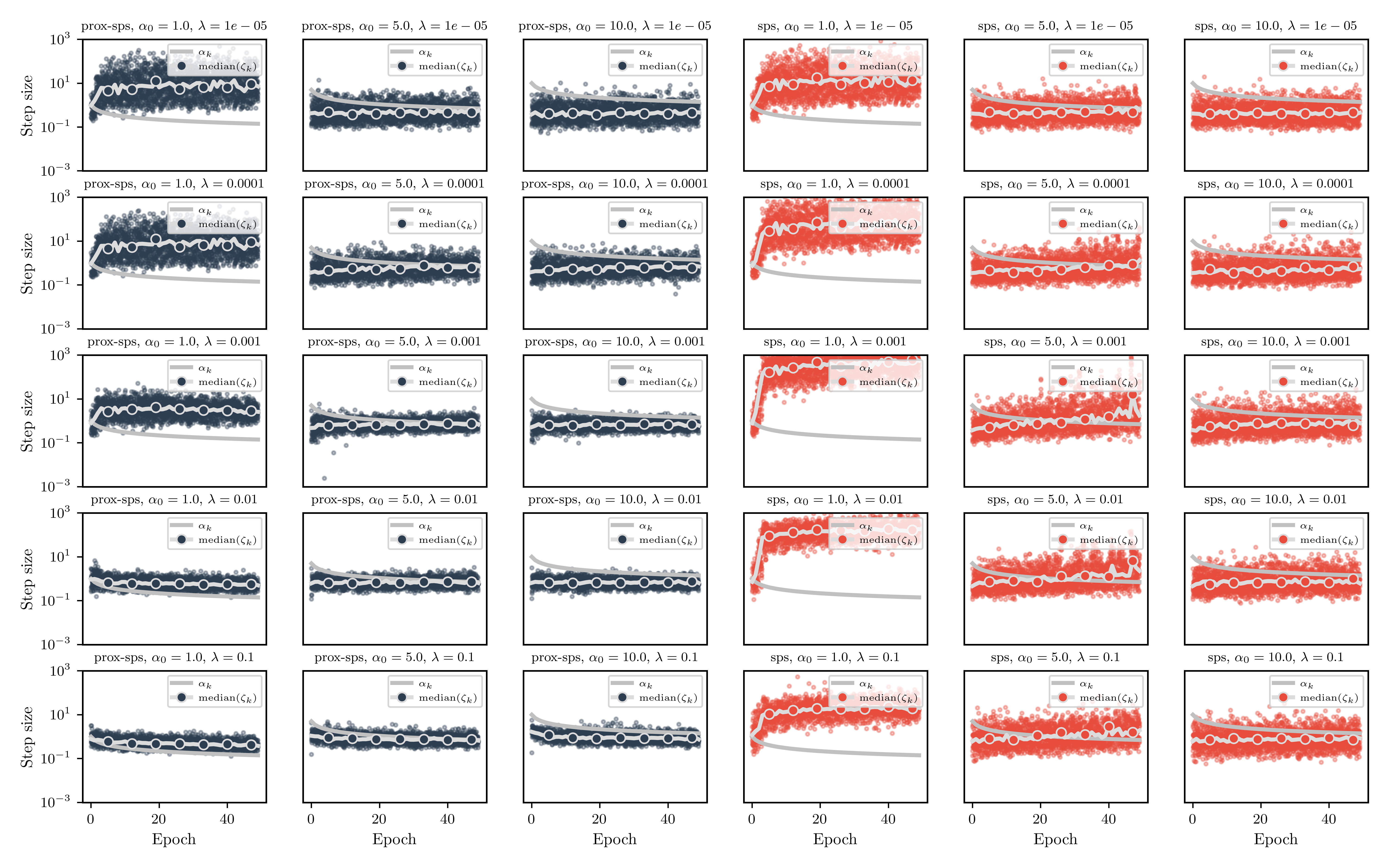}
    \caption{ Adaptive step size selection for \texttt{SPS} and \texttt{ProxSPS}. We plot $\zeta_k$ (see definition in \cref{sec:numerics-general}) as dots for each iteration as well as their median over each epoch. For this plot, we use the results of only one of the ten runs.}
    \label{fig:matrix_fac1_step_sizes}
\end{figure}
%
\subsection{Deep networks for image classification}\label{sec:deep-learning-exp}
%
We train a \texttt{ResNet56} and \texttt{ResNet110} model \citep{He2016} on the \texttt{CIFAR10} dataset. We use the data loading and preprocessing procedure and network implementation from \url{https://github.com/akamaster/pytorch_resnet_cifar10}. 
We do not use batch normalization.
The loss function is the cross-entropy loss of the true image class with respect to the predicted class probabilities, being the output of the \texttt{ResNet56} network. We add $\frac{\lambda}{2}\|x\|^2$ as regularization term, where $x$ consists of all learnable parameters of the model.
The \texttt{CIFAR10} dataset consists of 60,000 images, each of size $32\times32$, from ten different classes. We use the \texttt{PyTorch} split into 50,000 training and 10,000 test examples and use a batch size of 128. 
For \texttt{AdamW}, we set the weight decay parameter to $\lambda$ and set all other hyperparameters to its default.
We use the \texttt{AdamW}-implementation from \url{https://github.com/zhenxun-zhuang/AdamW-Scale-free} as it does not -- in contrast to the \texttt{Pytorch} implementation -- multiply the weight decay parameter with the learning rate, which leads to better comparability to \texttt{SPS} and \texttt{ProxSPS} for identical values of $\lambda$.
 For \texttt{SPS} and \texttt{ProxSPS} we use the \texttt{sqrt}-schedule and $\alpha_0=1$.
 We run each method repeatedly using (the same) three different seeds for the dataset shuffling.
\vspace{2mm}\\
\textbf{Discussion:} For \texttt{Resnet56}, from the bottom plot in \cref{fig:cifar10-resnet56-metrics}, we observe that both \texttt{SPS} and \texttt{ProxSPS} work well with \texttt{ProxSPS} leading to smaller weights. For $\lambda=5e-4$, the progress of \texttt{ProxSPS} stagnates after roughly 25 epochs. This can be explained by looking at the adaptive step size term $\zeta_k$ in \cref{fig:cifar10-resnet56-step-sizes}: as it decays over time we have $\tau_k^+ = \zeta_k \ll \alpha_k$.  Since every iteration of \texttt{ProxSPS} shrinks the weights by a  factor $\frac{1}{1+\alpha_k \lambda}$, this leads to a bias towards zero. This suggests that we should choose $\alpha_k$ roughly of the order of $\zeta_k$, for example by using the values of $\zeta_k$ from the previous epoch.
   
For the larger model \texttt{Resnet110} however, \texttt{SPS} does not make progress for a long time because the adaptive step size is very small (see \cref{fig:cifar10-resnet110-metrics} and \cref{fig:cifar10-resnet110-step-sizes}). \texttt{ProxSPS} does not share this issue and performs well after a few initial epochs. For larger values of $\lambda$, the training is also considerably faster than for \texttt{AdamW}. 
Generally, we observe that \texttt{ProxSPS} (and \texttt{SPS} for \texttt{Resnet56}) performs well in comparison to \texttt{AdamW}. This is achieved without extensive hyperparameter tuning (in particular this suggests that setting $c=1$ in \spsmax{} leads to good results and reduces tuning effort). 
%
\begin{figure}[t]
    \centering
    \includegraphics[trim=0 0.2cm 0 0.7cm, width=0.8\textwidth]{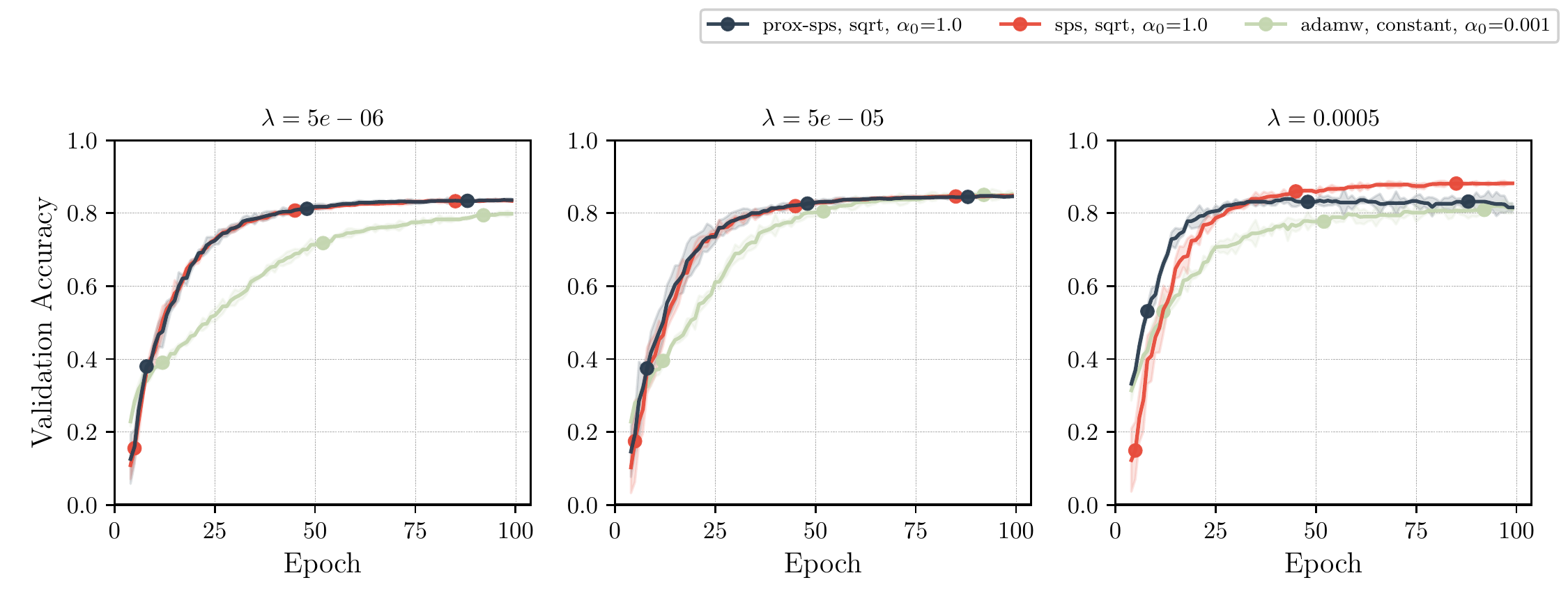}
    \includegraphics[trim=0 0.5cm 0 0.2cm, width=0.8\textwidth]{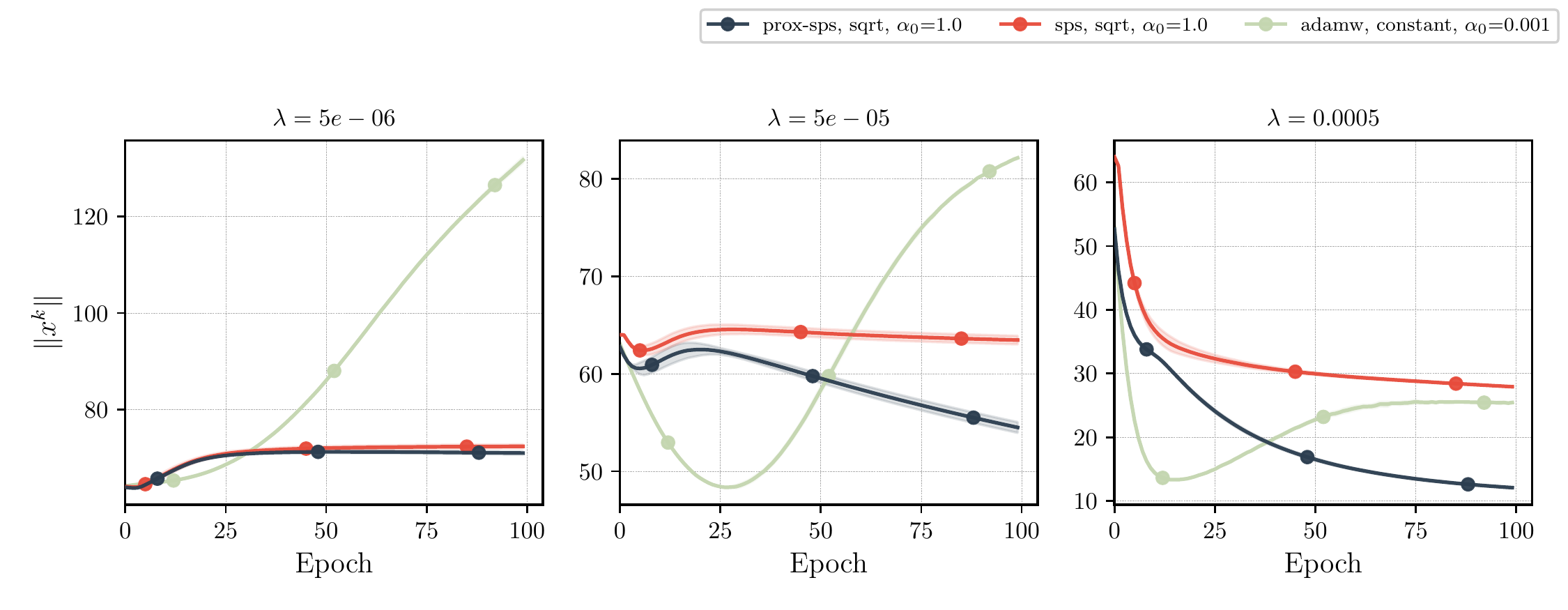}
    \caption[]
{\texttt{ResNet56}: 
(Top): Validation accuracy and model norm for three values of the regularization parameter $\lambda$. Validation accuracy is defined as the ratio of correctly labeled images on the validation set (i.e. \textit{Top-1 accuracy}), plotted as five-epoch running median. 
(Bottom): With $\|x^k\|$ we denote the norm of all learnable parameters at the $k$-th iteration.
Shaded area is two standard deviations over three independent runs.}
\label{fig:cifar10-resnet56-metrics}
\end{figure}
%
%
\begin{figure}[t]
    \centering
    \includegraphics[trim=0 0.2cm 0 0.7cm, width=0.8\textwidth]{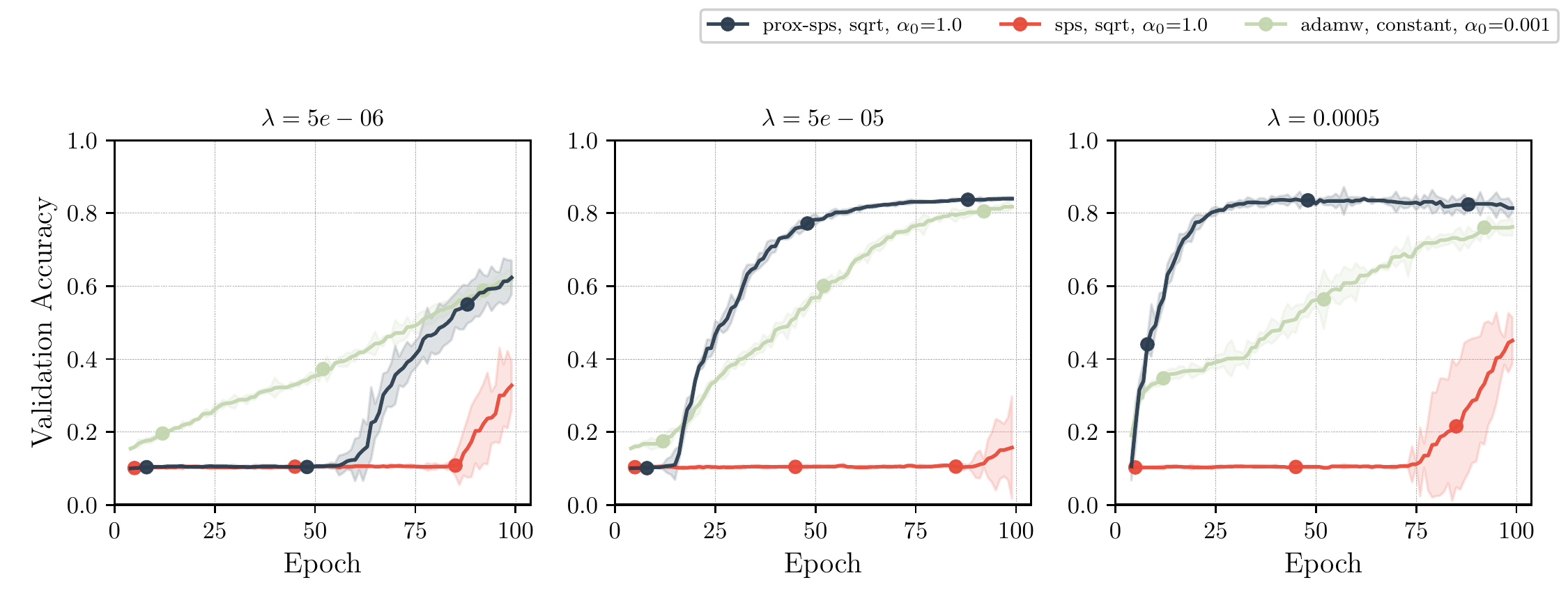}
    \includegraphics[trim=0 0.5cm 0 0.2cm, width=0.8\textwidth]{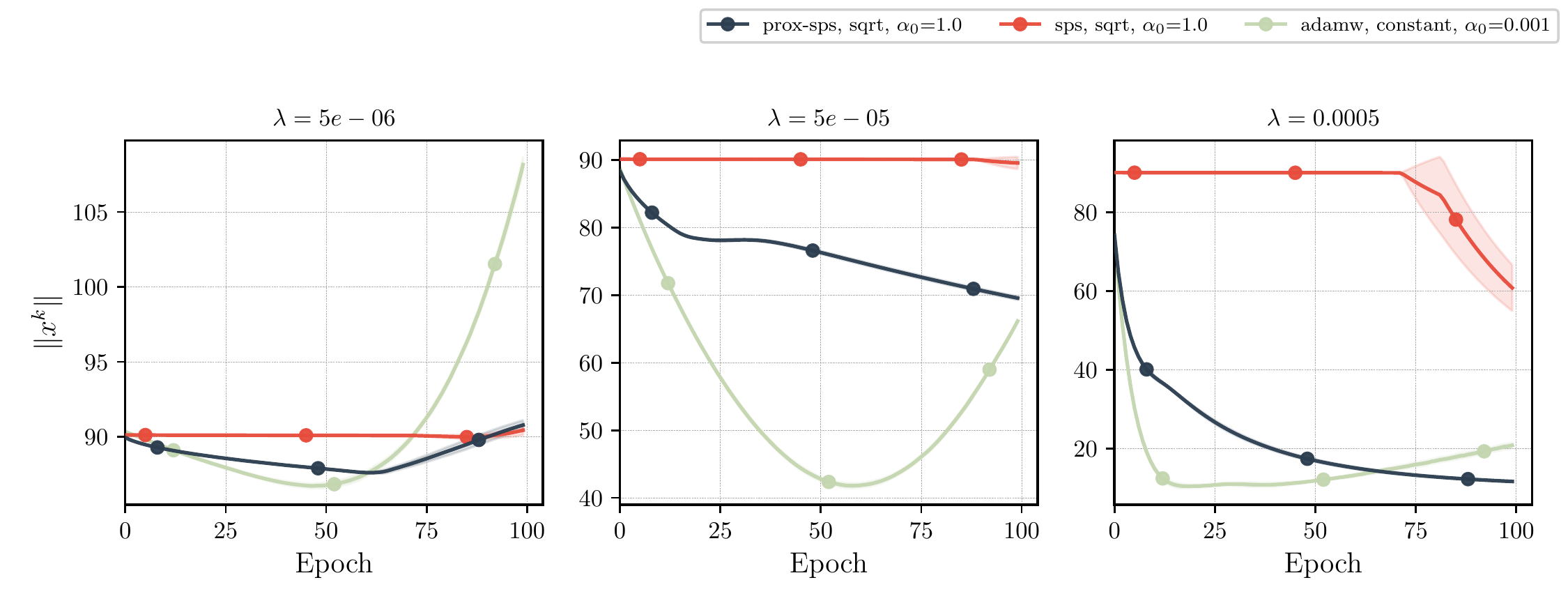}
    \caption[]
{\texttt{ResNet110}: Validation accuracy as five-epoch running median (top) and model norm (bottom) for three values of $\lambda$. Shaded area is two standard deviations over three independent runs.}
\label{fig:cifar10-resnet110-metrics}
\end{figure}
\begin{figure}[t]
\centering
\begin{subfigure}[t]{0.45\textwidth}
    \includegraphics[width=0.99\textwidth]{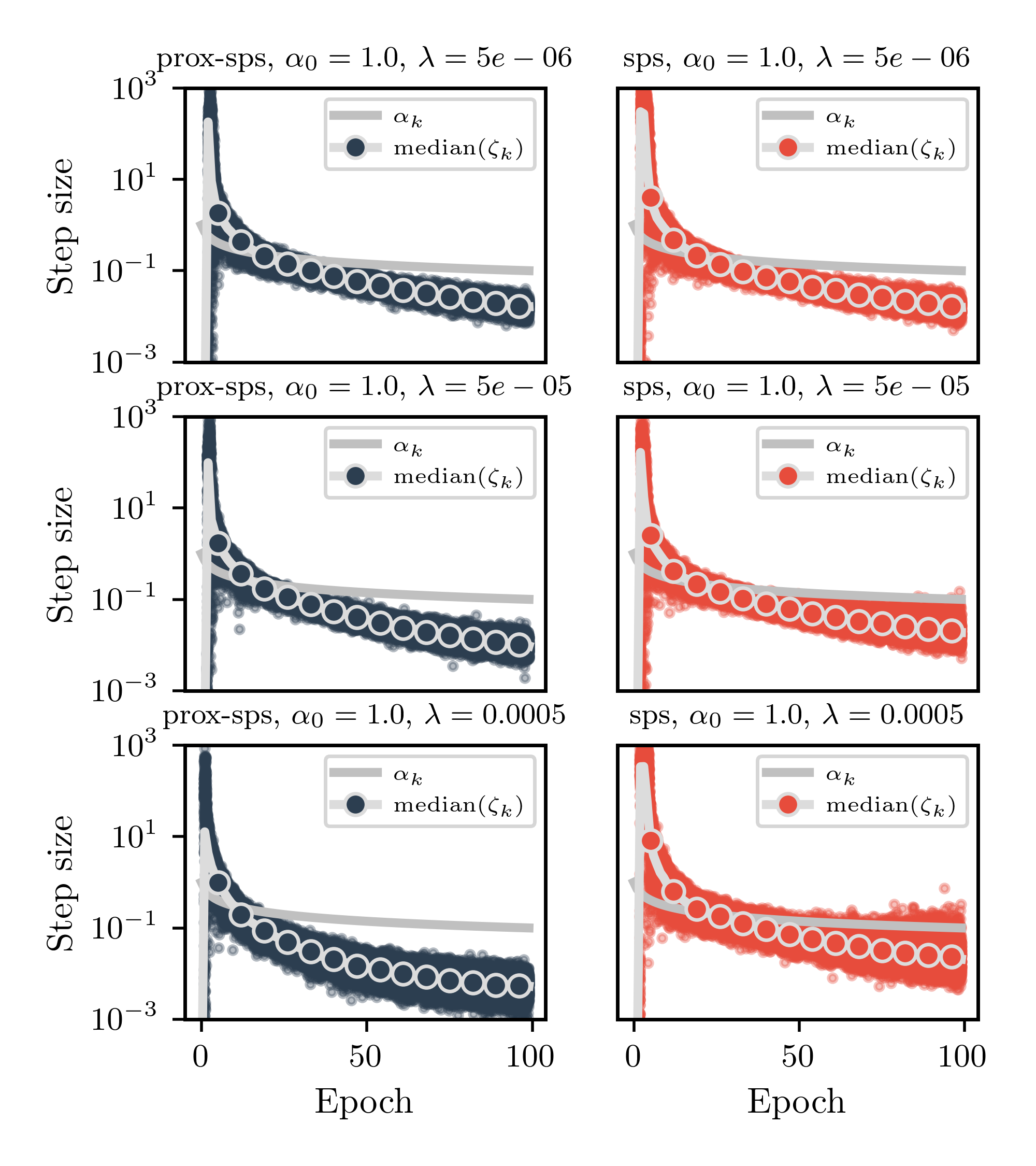}
    \caption{\texttt{ResNet56}}
    \label{fig:cifar10-resnet56-step-sizes}
\end{subfigure}
\begin{subfigure}[t]{0.45\textwidth}
    \includegraphics[width=0.99\textwidth]{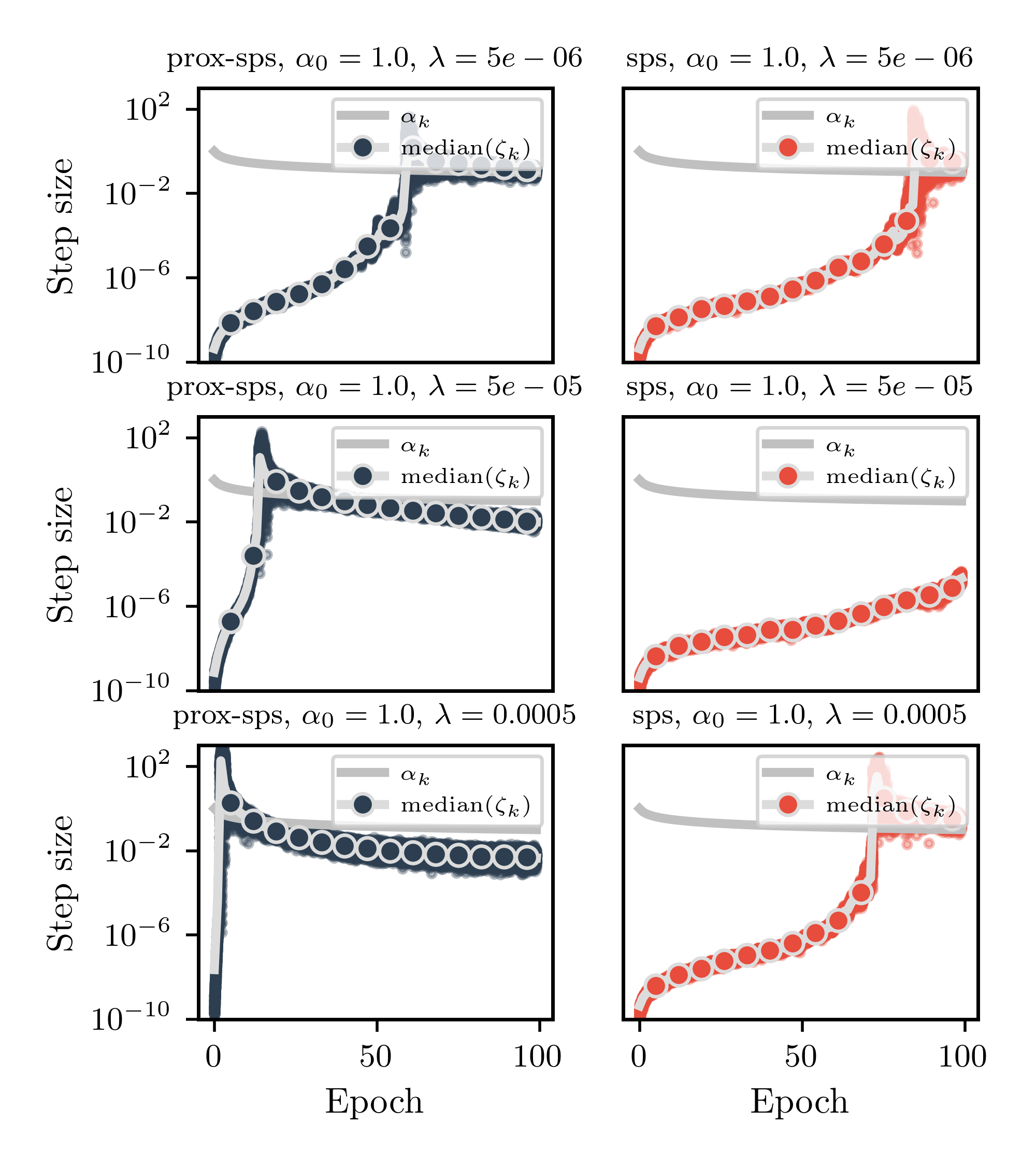}
    \caption{\texttt{ResNet110}}
    \label{fig:cifar10-resnet110-step-sizes}
\end{subfigure}
\caption{Adaptive step sizes for \texttt{SPS} and \texttt{ProxSPS}. See definition of $\zeta_k$ in \cref{sec:numerics-general}. For this plot, we use the results of only one of the three runs.}
\label{fig:cifar10-resnet-step-sizes}
\end{figure}    

Furthermore, we trained a \texttt{ResNet110} \emph{with} batch norm on the \texttt{Imagenet32} dataset. The plots and experimental details can be found in \cref{sec:imagenet32-exp}. From \cref{fig:imagenet32-resnet110-metrics}, we conclude that \texttt{SPS} and \texttt{ProxSPS} perform equally well in this experiment. Both \texttt{SPS} and \texttt{ProxSPS} are less sensititve with respect to the regularization parameter $\lambda$ than \texttt{AdamW} and the adaptive step size leads to faster learning in the initial epochs compared to \texttt{SGD}.
We remark that with batch norm, the effect of $\ell_2$-regularization is still unclear as the output of batch norm layers is invariant to scaling and regularization becomes ineffective \citep{Zhang2019}.

\section{Conclusion}
We proposed and  analyzed \texttt{ProxSPS}, a proximal version of the stochastic Polyak step size.  We arrived at \texttt{ProxSPS} by using the framework of  stochastic model-based proximal point methods. We then used this framework to argue that the resulting model of \texttt{ProxSPS} is a better approximation as compared to the model used by \texttt{SPS} when using regularization. 
Our theoretical results cover a wide range of optimization problems, including convex and nonconvex settings.  We performed a series of experiments comparing  \texttt{ProxSPS},  \texttt{SPS},  \texttt{SGD} and  \texttt{AdamW} when using $\ell_2$-regularization. In particular, we find that \texttt{SPS} can be very hard to tune when using  $\ell_2$-regularization, and in contrast,  \texttt{ProxSPS} performs well for a wide choice of step sizes and regularization parameters.  Finally, for our experiments on image classification, we find that \texttt{ProxSPS} is competitive to \texttt{AdamW}, whereas \texttt{SPS} can fail for larger models. At the same time \texttt{ProxSPS} produces smaller weights in the trained neural network. Having small weights may help reduce the memory footprint of the resulting network, and even suggests which weights can be pruned.  
\subsubsection*{Acknowledgments}
We thank the Simons Foundation for hosting Fabian Schaipp at the Flatiron Institute. We also thank the TUM Graduate Center for their financial support for the visit.
\bibliographystyle{tmlr}
\bibliography{lib}
\newpage
\tableofcontents

\appendix
\section{Missing Proofs} \label{sec:appendix-proofs}
\subsection{Proofs of model-based update formula}
\begin{lemma}
\label{lem:prox-sps-update-long}
For $\lambda \geq 0$, let $\phi(x)=\frac{\lambda}{2}\|x\|^2$ and let $g \in \partial f(x;s)$ and $C(s) \leq \inf_{z\in\R^n}f(z;s)$ hold for all $s\in\mathcal{S}$. For 
\begin{align*}
\psi_{x}(y;s) = f_x(y;s) + \phi(y), \quad f_x(y;s) = \max\{f(x;s) + \iprod{g}{y-x} , C(s)\},
\end{align*}
consider the update 
\begin{align}\label{prob-update-l2}
    x^{k+1} = \argmin_{x\in \R^n} \psi_{x^k}(x;S_k) + \frac{1}{2\alpha_k} \|x-x^k\|^2.
\end{align}
Denote $C_k:=C(S_k)$ and let $g_k\in \partial f(x^k;S_k)$. Define
\begin{align*}
    \tau_k^+ := \begin{cases} 0 \quad &\text{if } g_k=0,\\ 
    \min\left\{\alpha_k, \left(\frac{(1+\alpha_k\lambda)(f(x^k;S_k) - C_k) - \alpha_k\lambda \iprod{g_k}{x^k}}{\|g_k\|^2}\right)_+\right\} \quad &\text{else.}\end{cases}
\end{align*}
Then, we have
\begin{align}
\label{eqn:compact_update}
    x^{k+1} &= \frac{1}{1+\alpha_k\lambda} x^k - \frac{\tau_k^+}{1+\alpha_k\lambda}g_k = \frac{1}{1+\alpha_k\lambda} \Big(x^k - \tau_k^+g_k\Big) = \prox{\alpha_k \phi}(x^k - \tau_k^+ g_k).
\end{align}
Define $\tau_k := 0$ if $g_k=0$ and $\tau_k:= \min\left\{\alpha_k, \frac{(1+\alpha_k\lambda)(f(x^k;S_k) - C_k) - \alpha_k\lambda \iprod{g_k}{x^k}}{\|g_k\|^2}\right\}$ else. Then, it holds $\tau_k \leq \tau_k^+$ and 
\begin{align}
    \psi_{x^k}(x^{k+1}; S_k) = f(x^k;S_k) -\tfrac{\alpha_k\lambda}{1+\alpha_k\lambda} \iprod{g_k}{x^k} -\tfrac{\tau_k}{1+\alpha_k\lambda}\|g_k\|^2 + \phi(x^{k+1}). \label{eqn:f-val-reg-2}
\end{align}
\end{lemma}
%
\begin{proof}
Note that $\max\{f(x^k;S_k) + \iprod{g_k}{y-x^k} , C_k\}$ is convex as a function of $y$. The update is therefore unique. First, if $g_k=0$, then clearly $x^{k+1} = \prox{\alpha_k \phi}(x^k) = \tfrac{1}{1+\alpha_k \lambda}x^k$ and \eqref{eqn:f-val-reg-2} holds true. Now, let $g_k\neq 0$. The solution of \eqref{prob-update-l2} is either in $\{y \vert f(x^k;S_k) + \iprod{g_k}{y-x^k} < C_k\}$, or in $\{y \vert f(x^k;S_k) + \iprod{g_k}{y-x^k} > C_k\}$ or in $\{y \vert f(x^k;S_k) + \iprod{g_k}{y-x^k} = C_k\}$. We therefore solve three problems:
\begin{enumerate}[label=(P\arabic*)]
    \item Solve 
    \[y^+ = \argmin_y C_k + \frac{\lambda}{2}\|y\|^2 + \frac{1}{2\alpha_k}\|y-x^k\|^2. \]
    Clearly, the solution is $y^+ = \frac{1}{1+\alpha_k\lambda}x^k$. This $y^+$ solves \eqref{prob-update-l2} if $f(x^k;S_k) + \iprod{g_k}{y^+-x^k} < C_k$.
    \item Solve 
    \[y^+ = \argmin_y f(x^k;S_k) + \iprod{g_k}{y-x^k} + \frac{\lambda}{2}\|y\|^2 + \frac{1}{2\alpha_k}\|y-x^k\|^2. \]
    The optimality condition is $0= \alpha_k g_k + \alpha_k \lambda y^+ + y^+ -x^k$.
    Thus, the solution is $y^+ = \frac{1}{1+\alpha_k\lambda}(x^k - \alpha_k g_k)$. This $y^+$ solves \eqref{prob-update-l2} if $f(x^k;S_k) + \iprod{g_k}{y^+-x^k} > C_k$.
    \item Solve 
    \begin{align*}
        y^+ = \argmin_y \frac{\lambda}{2}\|y\|^2 + \frac{1}{2\alpha_k}\|y-x^k\|^2, \quad \text{s.t.\  } f(x^k;S_k) + \iprod{g_k}{y-x^k} = C_k. 
    \end{align*}
    The KKT conditions are given by 
    \begin{align*}
        \alpha_k \lambda y + y - x^k + \mu g_k &= 0, \\
        f(x^k;S_k) + \iprod{g_k}{y-x^k} &= C_k.
    \end{align*}
    Taking the inner product of the first equation with $g_k$, we get 
    \begin{align*}
        (1+\alpha_k \lambda) \iprod{g_k}{y} - \iprod{g_k}{x^k} + \mu \|g_k\|^2 = 0.
    \end{align*}
    From the second KKT condition we have $\iprod{g_k}{y} = C_k - f(x^k;S_k) + \iprod{g_k}{x^k}$, hence
    \[ (1+\alpha_k \lambda) \big(C_k-f(x^k;S_k) + \iprod{g_k}{x^k}\big)- \iprod{g_k}{x^k} + \mu \|g_k\|^2 = 0.\]
    Solving for $\mu$ gives $\mu= \frac{(1+\alpha_k\lambda) (f(x^k;S_k)-C_k) - \alpha_k\lambda \iprod{g_k}{x^k}}{\|g_k\|^2}$. From the first KKT condition, we obtain
    \[y^+ = \frac{1}{1+\alpha_k \lambda}\big(x^k - \mu g_k \big) =  \frac{1}{1+\alpha_k \lambda}\big(x^k - \frac{(1+\alpha_k\lambda) (f(x^k;S_k)-C_k) - \alpha_k\lambda \iprod{g_k}{x^k}}{\|g_k\|^2} g_k \big).\]
    This $y^+$ solves \eqref{prob-update-l2} if neither (P1) nor (P2) provided a solution.
\end{enumerate}
For all three cases, the solution takes the form $y^+ = \frac{1}{1+\alpha_k \lambda}[x^k - t g_k]=:y(t)$. As $\|g_k\|^2 > 0$, the term $f(x^k;S_k) + \iprod{g_k}{y(t)-x^k}$ is strictly monotonically decreasing in $t$. We know $f(x^k;S_k) + \iprod{g_k}{y(t)-x^k} = C_k$ for $t=\mu$ (from (P3)). Hence, $f(x^k;S_k) + \iprod{g_k}{y(t)-x^k} < C_k ~(>C_k)$ if and only if $t>\mu ~(t<\mu)$.

We conclude:
\begin{itemize}
    \item If $f(x^k;S_k) + \iprod{g_k}{y(0)-x^k} < C_k$, then the solution to (P1) is the solution to \eqref{prob-update-l2}. This condition is equivalent to $\mu <0$.
    \item If $f(x^k;S_k) + \iprod{g_k}{y(\alpha_k)-x^k} > C_k$, then the solution to (P2) is the solution to \eqref{prob-update-l2}. This condition is equivalent to $\alpha_k < \mu$.
    \item If neither $0 > \mu$ nor $\alpha_k < \mu$ hold, i.e.\ if $\mu \in [0, \alpha_k]$, then the solution to \eqref{prob-update-l2} comes from (P3) and hence is given by $y(\mu)$. 
\end{itemize}
Altogether, we get that $x^{k+1} = \frac{1}{1+\alpha_k \lambda}[x^k - \tau_k^+ g_k]$ with $\tau_k^+ = \min\{\alpha_k, (\mu)_+\}$.\\
Now, we prove \eqref{eqn:f-val-reg-2}. Note that if $g_k\neq0$, then $\tau_k = \min\{\alpha_k, \mu\}$ with $\mu$ defined as in (P3). In the case of (P1), we have $\psi_{x^k}(x^{k+1};S_k) = C_k + \phi(x^{k+1})$. Moreover, it holds $\mu <0$ and as $\alpha_k > 0$ we have $\tau_k = \mu$. Plugging $\tau_k=\mu$ into the right hand-side of \eqref{eqn:f-val-reg-2}, we obtain $C_k+ \phi(x^{k+1})$.\\
In the case of (P2) or (P3), we have $C_k \leq f(x^k;S_k) + \iprod{g_k}{x^{k+1}-x^k}$. Due to $f(x^k;S_k) + \iprod{g_k}{y(t)-x^k} = f(x^k;S_k) - \frac{1}{1+\alpha_k\lambda}\iprod{g_k}{x^k} +\frac{t}{1+\alpha_k\lambda}\|g_k\|^2$, we obtain \eqref{eqn:f-val-reg-2} as $x^{k+1} = y(\alpha_k)$ and $\mu>\alpha_k$ in the case of (P2) and $x^{k+1} = y(\mu)$ and $\mu\leq\alpha_k$ in the case of (P3).
\end{proof}
%
\begin{lemma}
\label{lem:sps-update-unreg}
Consider the model $f_x(y;s) := \max\{f(x;s) + \langle g, y-x\rangle, C(s)\}$ where $g\in \partial f(x;s)$ and $C(s)\leq \inf_{z\in\R^n} f(z;s)$ holds for all $s\in\mathcal{S}$. Then, update \eqref{eqn:model-spp-unreg} is given as 
\begin{align*}
    x^{k+1} = x^k - \gamma_k g_k, \quad \gamma_k = \begin{cases}0 \quad &\text{if } g_k=0,\\
    \min\Big\{\alpha_k, \frac{f(x^k;S^k) - C(S_k)}{\|g_k\|^2}\Big\} \quad &\text{else}.\end{cases}
\end{align*}
where $g_k \in \partial f(x^k;S_k)$. Moreover, it holds
\begin{align}
    \label{eqn:f-val-unreg}
    f_{x^k}(x^{k+1};S_k) = \max \{C(S_k), f(x^k;S_k) - \alpha_k \|g_k\|^2\},
\end{align}
and therefore $ f_{x^k}(x^{k+1};S_k) = f(x^k;S_k) - \gamma_k \|g_k\|^2$.
\end{lemma}
\begin{proof}
We apply \cref{lem:prox-sps-update-long} with $\lambda = 0$. As $f(x^k;S_k) \geq C(S_k)$, we have that $\tau_k^+ = \tau_k = \gamma_k$.
\end{proof}
\subsection{Proof of  \cref{thm:convex-smooth-reg} }\label{sec:proof-thm-conv}
From now on, denote with $\mathcal{F}_k$ the filtration that is generated by the history of all $S_j$ for $j=0,\dots,k-1$. 
\begin{proof}[Proof of \cref{thm:convex-smooth-reg}]
In the proof, we will denote $g_k=\nabla f(x^k;S_k)$.
We apply \cref{lem:smooth-basic-ineq}, \eqref{eqn:general-basic-ineq-reg} with $x=x^\star$. Due to \cref{lem:sps-model-properties} \ref{lem:sps-model-properties-ii} and convexity of $f(\cdot;s)$ it holds
\[\psi_{x^k}(x^\star;S_k) \leq f(x^\star;S_k)+\phi(x^\star).\]
Together with \eqref{eqn:general-model-value-reg}, we have
\begin{align}
\label{eqn:basic-ineq-reg}
\begin{split}
    (1+\alpha_k\lambda)\|x^{k+1}-x^\star\|^2 &\leq \|x^{k}-x^\star\|^2 - \|x^{k+1}-x^k\|^2 + 2\alpha_k[\phi(x^\star)-\phi(x^{k+1})] \\
    &\quad + 2\alpha_k \big[f(x^\star;S_k) - f(x^k;S_k) - \iprod{g_k}{x^{k+1}-x^k}\big].
\end{split}
\end{align}
Smoothness of $f$ yields
\begin{align*}
    -f(x^k) \leq -f(x^{k+1}) + \iprod{\nabla f(x^k)}{x^{k+1}-x^k} + \tfrac{L}{2}\|x^{k+1}-x^k\|^2.
\end{align*}
Consequently,
\begin{align*}
    &- \iprod{g_k}{x^{k+1}-x^k} = f(x^k) - f(x^k)  - \iprod{g_k}{x^{k+1}-x^k} \\
    & \hspace{2ex} \leq f(x^k) - f(x^{k+1}) + \iprod{\nabla f(x^k)-g_k}{x^{k+1}-x^k} + \tfrac{L}{2}\|x^{k+1}-x^k\|^2\\
    & \hspace{2ex} \leq f(x^k) - f(x^{k+1}) + \frac{\theta\alpha_k}{2}\|\nabla f(x^k)-g_k\|^2 + \frac{1}{2\theta\alpha_k} \|x^{k+1}-x^k\|^2 + \tfrac{L}{2}\|x^{k+1}-x^k\|^2. 
\end{align*}
for any $\theta>0$, where we used Young's inequality in the last step. Plugging into \eqref{eqn:basic-ineq-reg} gives
\begin{align*}
    (1+\alpha_k\lambda)\|x^{k+1}-x^\star\|^2 &\leq \|x^{k}-x^\star\|^2 + \big[\alpha_k L + \toneover{\theta} -1 \big] \|x^{k+1}-x^k\|^2 + 2\alpha_k[\phi(x^\star)-\phi(x^{k+1})] \\
    &\quad + 2\alpha_k \big[f(x^\star;S_k) -f(x^k;S_k) + f(x^k) - f(x^{k+1})\big] + \theta\alpha_k^2 \|\nabla f(x^k)-g_k\|^2.
\end{align*}
Applying conditional expectation, we have $\E[f(x^\star;S_k) \vert \mathcal{F}_k] = f(x^\star) $ and
\begin{align*}
    \E[-f(x^k;S_k) +f(x^k) \vert \mathcal{F}_k] = 0, \quad \E[\|\nabla f(x^k)-g_k\|^2 \vert \mathcal{F}_k] \leq \beta.
\end{align*}
Moreover, by assumption, $\alpha_k L + \toneover{\theta} -1 \leq 0$. Altogether, applying total expectation yields
\begin{align*}
    (1+\alpha_k\lambda)\E\|x^{k+1}-x^\star\|^2 \leq \E\|x^{k}-x^\star\|^2 + 2\alpha_k\E[\psi(x^\star)-\psi(x^{k+1})] + \theta\beta\alpha_k^2
\end{align*}
which proves \eqref{eqn:main-recursion-reg}.\\
\textbf{Proof of \ref{item-a:convex-smooth-reg}:} let $\alpha_k = \frac{1}{\lambda(k+k_0)}$. Denote $\Delta_k:= \E\|x^k-x^\star\|^2$. Rearranging and summing \eqref{eqn:main-recursion-reg}, we have
\begin{align*}
     \sum_{k=0}^{K-1}\E[\psi(x^{k+1})-\psi(x^\star)] \leq \sum_{k=0}^{K-1}\Big[ \tfrac{1}{2\alpha_k}\Delta_k - \tfrac{1+\alpha_k\lambda}{2\alpha_k}\Delta_{k+1}+ \tfrac{\theta\beta\alpha_k}{2}\Big].
\end{align*}
Plugging in $\alpha_k$, we have $\tfrac{1+\alpha_k\lambda}{2\alpha_k} = \tfrac{\lambda(k+k_0)}{2} + \tfrac{\lambda}{2}$ and thus 
\begin{align*}
     \sum_{k=0}^{K-1}\E[\psi(x^{k+1})-\psi(x^\star)] \leq \sum_{k=0}^{K-1}\Big[ \tfrac{\lambda(k+k_0)}{2}\Delta_k - \tfrac{\lambda(k+1+k_0)}{2}\Delta_{k+1} \Big] + \tfrac{\theta\beta}{2}\sum_{k=0}^{K-1}\tfrac{1}{\lambda(k+k_0)}.
\end{align*}
Dividing by $K$ and using convexity of $\psi$\footnote{By assumption $f$ is convex and therefore $\psi$ is convex.}, we have 
\begin{align*}
    \E\Big[\psi\Big(\tfrac{1}{K}\sum_{k=0}^{K-1}x^{k+1} \Big) -\psi(x^\star)\Big] \leq \frac{\lambda k_0}{2K}\|x^0-x^\star\|^2 + \tfrac{\theta\beta}{2\lambda K}\sum_{k=0}^{K-1}\tfrac{1}{k+k_0}.
\end{align*}
Finally, as $k_0\geq 1$, we estimate $\sum_{k=0}^{K-1}\tfrac{1}{k+k_0} \leq \sum_{k=0}^{K-1}\tfrac{1}{k+1} \leq 1+ \ln{K}$ by \cref{lem:bound-series} and obtain \eqref{eqn:strong-conv-decay-reg}.\\
\textbf{Proof of \ref{item-b:convex-smooth-reg}:}
Similar to the proof above, we rearrange and sum \eqref{eqn:main-recursion-reg} from $k=0,\dots,K-1$, and obtain
\begin{align*}
    \sum_{k=0}^{K-1}\alpha_k\E[\psi(x^{k+1})-\psi(x^\star)] \leq \frac{\|x^0-x^\star\|^2}{2}+ \frac{\theta\beta\sum_{k=0}^{K-1}\alpha_k^2}{2}.
\end{align*}
We divide by $\sum_{k=0}^{K-1}\alpha_k$ and use convexity of $\psi$ in order to obtain the left-hand side of \eqref{eqn:conv-decay-reg}. Moreover, by \cref{lem:bound-series} we have 
\[\sum_{k=0}^{K-1}\alpha_k \geq 2\alpha (\sqrt{K+1} -1), \quad \sum_{k=0}^{K-1}\alpha_k^2 \leq \alpha^2 (1+\ln K).\]
Plugging in the above estimates, gives
\begin{align*}
    \E\Big[\psi\Big(\tfrac{1}{\sum_{k=0}^{K-1}\alpha_k}\sum_{k=0}^{K-1}\alpha_k x^{k+1} \Big) -\psi(x^\star)\Big] \leq \frac{\|x^0-x^\star\|^2}{4\alpha(\sqrt{K+1}-1)} + \frac{\theta\beta\alpha (1+\ln K)}{4(\sqrt{K+1} -1)}.
\end{align*}
%
\textbf{Proof of \ref{item-c:convex-smooth-reg}:} If $f$ is $\mu$--strongly--convex, then $\psi$ is $(\lambda+\mu)$--strongly convex and 
\[\psi(x^\star)-\psi(x^{k+1}) \leq -\tfrac{\mu+\lambda}{2}\|x^{k+1}-x^\star\|^2.\]
From \eqref{eqn:main-recursion-reg}, with $\alpha_k = \alpha$,  we get
\begin{align*}
    (1+\alpha(\mu + 2 \lambda)) \E\|x^{k+1} - x^\star\|^2 \leq \E\|x^k - x^\star\|^2 + \theta \beta \alpha^2.
\end{align*}
Doing a recursion of the above from $k=0,\dots,K-1$ gives
\begin{align*}
    \E\|x^{K}-x^\star\|^2 \leq (1+\alpha(\mu+2\lambda))^{-K} \|x^0-x^\star\|^2 + \theta\beta\alpha^2 \sum_{k=1}^{K}(1+\alpha(\mu+2\lambda))^{-k}
\end{align*}
Using the geometric series, $\sum_{k=1}^{K}(1+\alpha(\mu+2\lambda))^{-k} \leq \frac{1+\alpha(\mu+2\lambda)}{\alpha(\mu+2\lambda)} - 1 = \frac{1}{\alpha(\mu+2\lambda)}$, and thus
\begin{align*}
    \E\|x^{K}-x^\star\|^2 \leq (1+\alpha(\mu+2\lambda))^{-K} \|x^0-x^\star\|^2 + \frac{\theta\beta\alpha}{\mu+2\lambda}.
\end{align*}
\end{proof}
%
%
\subsection{Proof of \cref{thm:exact-nonconv-reg}}\label{sec:proof-thm-nonconv}
\begin{proof}[Proof of \cref{thm:exact-nonconv-reg}]
In the proof, we will denote $g_k=\nabla f(x^k;S_k)$. By assumption $f$ is $\rho$-weakly convex and hence $\psi$ is $(\rho-\lambda)$-weakly convex if $\rho> \lambda$ and convex if $\rho \leq \lambda$. Hence, $\hat x^k := \prox{\eta \psi}(x^k)$ is well-defined for $\eta < 1/(\rho-\lambda)$ if $\rho>\lambda$ and for any $\eta>0$ else. Note that $\hat x^k$ is $\mathcal{F}_k$--measurable. 
We apply \cref{lem:smooth-basic-ineq}, \eqref{eqn:general-basic-ineq-reg} with $x=\hat{x}^k$. Due to \cref{lem:sps-model-properties} \ref{lem:sps-model-properties-ii} it holds
\[\psi_{x^k}(\hat x^k;S_k) = f_{x^k}(\hat x^k;S_k) +\phi(\hat{x}^k) \leq f(\hat x^k;S_k) + \tfrac{\rho_{S_k}}{2}\|\hat x^k - x^k\|^2 + \phi(\hat{x}^k). \]
Together with \eqref{eqn:general-model-value-reg}, this gives
\begin{align*}
    (1+\alpha_k \lambda)\|x^{k+1}-\hat x^k\|^2 \leq &(1+\alpha_k \rho_{S_k})\|x^{k}-\hat x^k\|^2  - \|x^{k+1}-x^k\|^2 \\
    &+ 2\alpha_k \Big(\phi(\hat x^k) - \phi(x^{k+1}) + f(\hat x^k;S_k) - f(x^k;S_k) - \iprod{g_k}{x^{k+1}-x^k}\Big)
\end{align*}
Analogous to the proof of \cref{thm:convex-smooth-reg}, due to Lipschitz smoothness, for all $\theta > 0$ we have
\begin{align*}
    -f(x^k;S_k) - \iprod{g_k}{x^{k+1}-x^k} &\leq -f(x^k;S_k) +f(x^k) \\
    & \quad - f(x^{k+1}) + \tfrac{\theta\alpha_k}{2}\|\nabla f(x^k) - g_k\|^2 + \big[\tfrac{1}{2\theta\alpha_k} + \tfrac{L}{2}\big]\|x^{k+1}-x^k\|^2.
\end{align*}
Plugging in gives
\begin{align*}
    &(1+\alpha_k \lambda)\|x^{k+1}-\hat x^k\|^2 \leq (1+\alpha_k \rho_{S_k})\|x^{k}-\hat x^k\|^2 + 2\alpha_k \Big(\phi(\hat x^k) - \phi(x^{k+1})\Big) \\
    & \hspace{2ex} + 2\alpha_k\big(f(\hat{x}^k;S_k)-f(x^k;S_k) + f(x^k) - f(x^{k+1}) + \tfrac{\theta\alpha_k}{2}\|\nabla f(x^k) - g_k\|^2\big) \\
    & \hspace{2ex} + \big[ \tfrac{1}{\theta} + \alpha_k L -1 \big]\|x^{k+1}-x^k\|^2.
\end{align*}
It holds $\E[f(\hat x^k;S_k) - f(x^k;S_k) \vert \mathcal{F}_k] = f(\hat x^k)-f(x^k)$ and $\E[\psi(\hat x^k)\vert \mathcal{F}_k] = \psi(\hat x^k)$. By \cref{asum:gradient-noise}, we have
$\E[ \|g_k - \nabla f(x^k)\|^2 \vert \mathcal{F}_k] \leq \beta.$
Altogether, taking conditional expectation yields
\begin{align*}
    &(1+\alpha_k \lambda)\E[ \|x^{k+1}-\hat x^k\|^2\vert \mathcal{F}_k] \leq (1+\alpha_k\rho) \|x^{k}-\hat x^k\|^2 + 2\alpha_k \E\big[\psi(\hat x^k) - \psi(x^{k+1}) \vert \mathcal{F}_k\big] \\
    & \hspace{2ex} + \alpha_k^2 \theta \beta + \big[ \tfrac{1}{\theta} + \alpha_k L - 1 \big]\E[\|x^{k+1}-x^k\|^2\vert \mathcal{F}_k].
\end{align*}
Next, the definition of the proximal operator implies that almost surely
\begin{align*}
    \psi(\hat x^k) +\tfrac{1}{2\eta}\|\hat x^k - x^k\|^2  \leq \psi(x^{k+1}) + \tfrac{1}{2\eta}\|x^{k+1}-x^k\|^2,
\end{align*}
and hence
\begin{align*}
    \E\big[\psi(\hat x^k) - \psi(x^{k+1}) \vert \mathcal{F}_k\big] \leq \E\big[\tfrac{1}{2\eta}\|x^{k+1}-x^k\|^2 - \tfrac{1}{2\eta}\|\hat x^k - x^k\|^2 \vert \mathcal{F}_k\big].
\end{align*}
Altogether, we have
\begin{align*}
    (1+\alpha_k \lambda)\E[ \|x^{k+1}-\hat x^k\|^2\vert \mathcal{F}_k] &\leq (1+\alpha_k(\rho-\eta^{-1})) \|x^{k}-\hat x^k\|^2 \\
    & \quad + \alpha_k^2 \theta \beta + \big[ \tfrac{1}{\theta} + \alpha_k L +\alpha_k\eta^{-1} - 1 \big]\E[\|x^{k+1}-x^k\|^2\vert \mathcal{F}_k].
\end{align*}
From assumption \eqref{eqn:cond-eta-alpha-reg}, we can drop the last term.
Now, we aim for a recursion in $\mathrm{env}^\eta_\psi$. Using that 
\begin{align*}
    \frac{1+\alpha_k(\rho-\eta^{-1})}{1+\alpha_k\lambda} = \frac{1+\alpha_k\lambda-\alpha_k\lambda+\alpha_k(\rho-\eta^{-1})}{1+\alpha_k\lambda} = 1+\frac{\alpha_k(\rho-\eta^{-1}-\lambda)}{1+\alpha_k\lambda} \leq 1+\alpha_k(\rho-\eta^{-1}-\lambda),
\end{align*}
we get
\begin{align*}
     \E[\mathrm{env}^\eta_\psi(x^{k+1}) \vert \mathcal{F}_k] &\leq \E[\psi(\hat x^k) + \frac{1}{2\eta} \|x^{k+1}-\hat x^k\|^2\vert \mathcal{F}_k]\\
     & \leq \underbrace{\psi(\hat x^k) + \frac{1}{2\eta}\|x^{k}-\hat x^k\|^2}_{=\mathrm{env}^\eta_\psi(x^k)} + \frac{1}{2\eta}\big[\alpha_k(\rho-\eta^{-1}-\lambda)\big]\|x^{k}-\hat x^k\|^2 + \frac{\alpha_k^2}{2\eta} \theta \beta.
\end{align*}
Now using $\|x^{k}-\hat x^k\| = \eta\|\nabla \mathrm{env}^\eta_\psi(x^k)\|$ we conclude
\begin{align*}
    &\E[\mathrm{env}^\eta_\psi(x^{k+1}) \vert \mathcal{F}_k] \leq \mathrm{env}^\eta_\psi(x^k) + \frac{\eta}{2}\big[\alpha_k(\rho-\eta^{-1}-\lambda)\big] \|\nabla \mathrm{env}^\eta_\psi(x^k)\|^2
     + \frac{\alpha_k^2}{2\eta} \theta \beta.
\end{align*}
Due to \eqref{eqn:cond-eta-alpha-reg}, we have $\eta^{-1} +\lambda - \rho > 0$. Taking expectation and unfolding the recursion by summing over $k=0,\dots,K-1$, we get
\begin{align*}
    \sum_{k=0}^{K-1}\tfrac{\alpha_k}{2}(1-\eta(\rho-\lambda)) \E \|\nabla \mathrm{env}^\eta_\psi(x^k)\|^2 \leq \mathrm{env}^\eta_\psi(x^0) - \E[\mathrm{env}^\eta_\psi(x^K)] + \sum_{k=0}^{K-1}\frac{\alpha_k^2}{2\eta} \theta \beta.
\end{align*}
Now using that $\mathrm{env}^\eta_\psi(x^K) \geq \inf \psi$ almost surely, we finally get 
\begin{align}
    \label{eqn:moreau-env-recurse-reg-proof}
    \sum_{k=0}^{K-1}\alpha_k \E \|\nabla \mathrm{env}^\eta_\psi(x^k)\|^2 \leq \frac{2(\mathrm{env}^\eta_\psi(x^0) -\inf \psi)}{1-\eta(\rho-\lambda)}+ \frac{\beta\theta}{\eta(1-\eta(\rho-\lambda))}\sum_{k=0}^{K-1}\alpha_k^2,
\end{align}
which proves \eqref{eqn:moreau-env-recurse-reg}.
Now choose $\alpha_k = \frac{\alpha}{\sqrt{k+1}}$ and divide \eqref{eqn:moreau-env-recurse-reg-proof} by $\sum_{k=0}^{K-1} \alpha_k$. Using \cref{lem:bound-series} for $\sum_{k=0}^{K-1} \alpha_k$ and $\sum_{k=0}^{K-1} \alpha_k^2$, we have
\begin{align*}
    \min_{k=0,\dots,K-1}\E \|\nabla \mathrm{env}^\eta_\psi(x^k)\|^2 \leq \frac{\mathrm{env}^\eta_\psi(x^0) -\inf \psi}{\alpha(1-\eta(\rho-\lambda))(\sqrt{K+1}-1)} + \frac{\beta\theta}{2\eta(1-\eta(\rho-\lambda))}\frac{\alpha(1+\ln K)}{(\sqrt{K+1}-1)}.
\end{align*}
Choosing $\alpha_k = \frac{\alpha}{\sqrt{K}}$ instead, we can identify the left-hand-side of \eqref{eqn:moreau-env-recurse-reg-proof} as $\alpha \sqrt{K} \E \|\nabla \mathrm{env}^\eta_\psi(x^K_\sim)\|^2$. Dividing by $\alpha \sqrt{K}$ and using $\sum_{k=0}^{K-1} \alpha_k^2 = \alpha^2$, we obtain
\begin{align*}
    \E \|\nabla \mathrm{env}^\eta_\psi(x^K_\sim)\|^2 \leq \frac{2(\mathrm{env}^\eta_\psi(x^0) -\inf \psi)}{\alpha(1-\eta(\rho-\lambda))\sqrt{K}} + \frac{\beta\theta}{\eta(1-\eta(\rho-\lambda))}\frac{\alpha}{\sqrt{K}}.
\end{align*}
\end{proof}
%
\section{Auxiliary Lemmas}
%
\begin{lemma}[Thm.\ 4.5 in \citep{Drusvyatskiy2019}]	
\label{lem:env-gradient-mapping}
Let $f$ be $L$-smooth and $\phi$ be proper, closed, convex. For $\eta > 0$, define $\mathcal{G}_\eta(x) := \eta^{-1}\big(x- \prox{\eta \phi}(x-\eta \nabla f(x))\big)$. It holds
\[\tfrac{1}{4}\|\nabla \env{\psi}^{1/(2L)}(x)\| \leq \|\mathcal{G}_{1/L}(x)\| \leq \tfrac{3}{2}(1+\toneover{\sqrt{2}})\|\nabla \env{\psi}^{1/(2L)}(x)\| \quad \forall x \in \R^n.\]
\end{lemma}
\begin{lemma}\label{lem:update-prob-general}
Let $c\in\R, a, x^0\in\R^n$ and $\beta>0$ and let $\phi:\R^n\to \R\cup \{\infty\}$ be proper, closed, convex. The solution to 
\begin{align}\label{eqn:update-prob-general}
    y^+ = \argmin_{y\in \R^n}\quad  \big(c+\iprod{a}{y}\big)_+ +\phi(y)+ \frac{1}{2\beta}\|y-x^0\|^2
\end{align}
is given by 
\begin{align}\label{eqn:update-formula-general}
    y^+ = \begin{cases} 
    \prox{\beta \phi}(x^0-\beta a), \quad &\text{if } c+\iprod{a}{\prox{\beta \phi}(x^0-\beta a)} > 0, \\
    \prox{\beta \phi}(x^0), \quad &\text{if } c+\iprod{a}{\prox{\beta \phi}(x^0)} < 0, \\
    \prox{\beta \phi}(x^0-\beta u a) \quad &\text{else, for } u\in[0,1] \text{ such that } c+\iprod{a}{\prox{\beta \phi}(x^0-\beta u a)} = 0.
    \end{cases}
\end{align}
\end{lemma}
\begin{remark}
The first two conditions can not hold simultaneously due to uniqueness of the solution. If neither of the conditions of the first two cases are satisfied, we have to find the root of $u\mapsto c+\iprod{a}{\prox{\beta \phi}(x^0-\beta u a)}$ for $u\in[0,1]$. Due to strong convexity of the objective in \eqref{eqn:update-prob-general}, we know that there exists a root and hence $y^+$ can be found efficiently with bisection.
\end{remark}
\begin{proof}
The objective of \eqref{eqn:update-prob-general} is strongly convex and hence there exists a unique solution. Due to \citep[Thm.\ 3.63]{Beck2017}, $y$ is the solution to \eqref{eqn:update-prob-general} if and only if it satisfies first-order optimality, i.e.
\begin{align}\label{eqn:foc-update}
    \exists u\in \partial(\cdot)_+(c+\iprod{a}{y}): ~ 0\in ua + \partial \phi(y) + \frac{1}{\beta}(y-x^0).
\end{align}
Now, as $y=\prox{\beta \phi}(z) \iff 0 \in \partial \phi(y) + \frac{1}{\beta}(y-z)$, it holds
\begin{align*}
    \eqref{eqn:foc-update} &\iff \exists u\in \partial(\cdot)_+(c+\iprod{a}{y}): ~0 \in \partial \phi(y) + \frac{1}{\beta}(y-(x^0-\beta ua ))\\
    &\iff \exists u\in \partial(\cdot)_+(c+\iprod{a}{y}): ~y=\prox{\beta \phi}(x^0-\beta u a).
\end{align*}
We distinguish three cases:
\begin{enumerate}
    \item Let $\bar y := \prox{\beta \phi}(x^0-\beta a)$ and suppose that $c+\iprod{a}{\bar y}> 0$. Then $\partial(\cdot)_+(c+\iprod{a}{\bar y}) =\{1\}$ and hence $\bar y$ satisfies \eqref{eqn:foc-update} with $u=1$. Hence, $y^+=\bar y$.
    \item  Let $\bar y := \prox{\beta \phi}(x^0)$ and suppose that $c+\iprod{a}{\bar y}< 0$. Then $\partial(\cdot)_+(c+\iprod{a}{\bar y}) =\{0\}$ and hence $\bar y$ satisfies \eqref{eqn:foc-update} with $u=0$. Hence, $y^+=\bar y$.
    \item If neither the condition of the first nor of the second case of \eqref{eqn:update-formula-general} are satisfied, then, as \eqref{eqn:foc-update} is a necessary condition for the solution $y^+$, it must hold $c+\iprod{a}{y^+} = 0$. 
    Hence, there exists a $u\in \partial(\cdot)_+(c+\iprod{a}{y^+}) = [0,1]$ such that
    \[c+\iprod{a}{\prox{\beta \phi}(x^0-u\beta a)} = 0.\]
\end{enumerate}
\end{proof}
\begin{lemma}\label{lem:bound-series}
For any $K\geq 1$ it holds
\begin{align*}
    &\sum_{k=0}^{K-1} \tfrac{1}{k+1} = 1+ \sum_{k=1}^{K-1} \tfrac{1}{k+1} \leq 1+\int_0^{K-1} \tfrac{1}{s+1}ds = 1+\ln K,\\
    &\sum_{k=0}^{K-1} \tfrac{1}{\sqrt{k+1}} \geq \int_0^{K} \tfrac{1}{\sqrt{s+1}}ds = 2\sqrt{K+1} -  2.
\end{align*}
\end{lemma}

The following is a detailled version of \cref{prop:one-sided-model-reg}. We refer to \cref{sec:proxsps-theory} for context.
\begin{proposition} \label{prop:one-sided-model-reg-appendix}
Let \cref{assum:inf-samples} and \cref{asum:phi} hold and assume that there is an open, convex set $U$ containing $\mathrm{dom}~\phi$.
Let $f(\cdot;s)$ be $\rho_s$--weakly convex for all $s\in\mathcal{S}$ and let $\rho=\E[\rho_S]$. 
Assume that there exists $G_s \in \R_+$ for all $s\in \mathcal{S}$, such that $\mathsf{G}:=\sqrt{\E[G_S^2]}< \infty$ and 
\begin{align}\label{eqn:glob-bounded-grads-appendix}
        \|g(x;s)\| \leq G_s \quad \forall g(x;s) \in \partial f(x;s),~ \forall x \in U.
\end{align}
Then, $\psi_{x}(y;s)$ (given in \eqref{eqn:model-sps-prox}) satisfies the following:
\begin{enumerate}[label=(B\arabic*)]
    \item It is possible to generate infinitely many i.i.d.\ realizations $S_1,S_2,\dots$ from $\mathcal{S}$.
    \item It holds $\E[f_x(x;S)] = f(x)$ and $\E[f_x(y;S)] \leq f(y) + \frac{\rho}{2}\|y-x\|^2$ for all $x,y\in \R^n$.
    \item The mapping $\psi_x(\cdot;s)=f_x(\cdot;s)+\phi(\cdot)$ is convex for all $x\in \R^n$ and all $s\in \mathcal{S}$.
    \item For all $x,y \in U$ and $s\in \mathcal{S}$, it holds
    $f_x(x;s) - f_x(y;s) \leq G_s \|x-y\|. $
\end{enumerate}
\end{proposition}
\begin{proof}
The properties (B1)--(B4) are identical to (B1)--(B4) in \citep[Assum.\ B]{Davis2019}, setting $r=\phi$, $f_x(\cdot,\xi) = f_x(\cdot;s)$, $\eta= 0$, $\tau = \rho$, $\mathsf{L} = \mathsf{G}$, and $L(\xi) = G_s$.
(B1) is identical to \cref{assum:inf-samples}. (B2) holds due to \cref{lem:sps-model-properties}, \ref{lem:sps-model-properties-ii}, applying expectation and using the definition of $f$, i.e.\ $f(x) = \E[f(x;S)]$. (B3) holds due to \cref{lem:sps-model-properties}, \ref{lem:sps-model-properties-i} and convexity of $\phi$.
For (B4), taking $g\in \partial f(x;s)$ and $x,y\in U$, we have 
\[f_{x}(x;s) - f_{x}(y;s)  \leq f(x;s) - f(x;s) - \iprod{g}{y-x} \leq \|g\|\|y-x\| \leq G_s \|x-y\|.\]
\end{proof}

%
\section{Model equivalence for \texttt{SGD}} \label{sec:sgd-equiv}
In the unregularized case, the \texttt{SGD} update 
\[x^{k+1}= x^k-\alpha_k g_k, \quad g_k \in \partial f(x^k;S_k),\]
can be seen as solving \eqref{eqn:model-spp-unreg} with the model
\[f_x(y;s) = f(x;s)+ \iprod{g}{y-x}, \quad g\in \partial f(x;s).\]
Now, consider again the regularized problem \eqref{prob:composite} with $\phi(x)=\frac{\lambda}{2}\|x\|^2$ and update \eqref{eqn:model-spp-reg} . \\
On the one hand, the model $\psi_x(y;s) = f(x;s)+\phi(x)+ \iprod{g+\lambda x}{y-x}$ with $g\in \partial f(x;s)$ yields
\begin{align}\label{eqn:sgd-naive}
x^{k+1} &= x^k - \alpha_k (g_k + \lambda x^k) = (1-\alpha_k \lambda)x^k - \alpha_k g_k.
\end{align}
On the other hand, the model $\psi_x(y;s) = f(x;s)+ \iprod{g}{y-x} + \phi(y)$ with $g\in \partial f(x;s)$ results in
\begin{align}\label{eqn:sgd-prox}
x^{k+1} &= \prox{\alpha_k \phi}(x^k-\alpha_k g_k) = \oneover{1+\alpha_k \lambda}\big[x^k - \alpha_k g_k \big] = (1-\frac{\alpha_k}{1+\alpha_k \lambda}\lambda)x^k - \frac{\alpha_k}{1+\alpha_k \lambda} g_k.
\end{align}
Running \eqref{eqn:sgd-naive} with step sizes $\alpha_k=\beta_k$ is equivalent to running \eqref{eqn:sgd-prox} with step sizes $\frac{\alpha_k}{1+\alpha_k \lambda} = \beta_k \iff \alpha_k = \frac{\beta_k}{1-\beta_k\lambda}$. In this sense, standard \texttt{SGD} can be seen to be equivalent to proximal \texttt{SGD} for $\ell_2$--regularized problems.
\section{Additional information on numerical experiments}
\subsection{Matrix Factorization}\label{sec:appendix-numerics-matrixfac}
\textbf{Synthetic data generation:} We consider the experimental setting of the deep matrix factorization experiments in \citep{Loizou2021}, but with an additional regularization. We generate data in the following way: first sample $B\in \R^{q\times p}$ with uniform entries in the interval $[0,1]$. Then choose $\upsilon\in\R$ (which will be our targeted inverse  condition number) and compute $A=DB$ where $D$ is a diagonal matrix with entries from $1$ to $\upsilon$ (equidistant on a logarithmic scale)\footnote{Note that \citep{Loizou2021} uses entries from $1$ to $\upsilon$ on a \textit{linear} scale which, in our experiments, did not result in large condition numbers even if $\upsilon$ is very small.}. In order to investigate the impact of regularization, we generate a noise matrix $E$ with uniform entries in $[-\epsilon,\epsilon]$ and set $\tilde{A} := A \odot (1+E)$. We then sample $y^{(i)} \sim N(0,I)$ and compute the targets $b^{(i)} = \tilde{A} y^{(i)}$.
A validation set of identical size is created by the same mechanism, but computing its targets, denoted by $b^{(i)}_\text{val}$, via the original matrix $A$ instead of $\tilde{A}$. The validation set contains $N_{\text{val}}=N$ samples. 
\begin{table}[h]
\centering
\begin{tabular}{ |ccccccc| } 
 \hline
 Name & $p$ & $q$ & $N$ & $\upsilon$ & $r$ & $\epsilon$ \\ [0.5ex] 
 \hline
\texttt{matrix-fac1} & 6 & 10 & 1000 & \texttt{1e-5} & 4 & 0\\ 
\texttt{matrix-fac2} & 6 & 10 & 1000 & \texttt{1e-5} & 10 & 0.05\\ 
 \hline
\end{tabular}
\caption{Matrix factorization synthetic datasets.}
\label{table:matrix-fac}
\end{table}
\vspace{2mm}\\
\textbf{Model and general setup:} Problem \eqref{prob-matrix-fac-reg} can be interpreted as a two-layer neural network without activation functions. We train the network using the squared distance of the model output and $b^{(i)}$ (averaged over a mini-batch) as the loss function. We run 50 epochs for different methods, step size schedules and values of $\lambda$. For each different instance, we do ten independent runs: each run has the identical training set and initialization of $W_1$ and $W_2$, but different shuffling of the training set and different samples $y^{(i)}$ for the validation set. In order to allow a fair comparison, all methods have identical train and validation sets across all runs. All metrics are averaged over the ten runs. We always use a batch size of $20$.
\subsection{Plots for \texttt{matrix-fac2}} \label{sec:matrix-fac2-plots}
In this section, we plot additional results for Matrix Factorization, namely for the setting \texttt{matrix-fac2} of \cref{table:matrix-fac}, see \cref{fig:matrix_fac2_stability},~\cref{fig:matrix_fac2_stability_val}, and \cref{fig:matrix_fac2}. The results are qualitatively very similar to the setting \texttt{matrix-fac1}.
\begin{figure}[H]
    \centering
    \includegraphics[height=0.2\textheight]{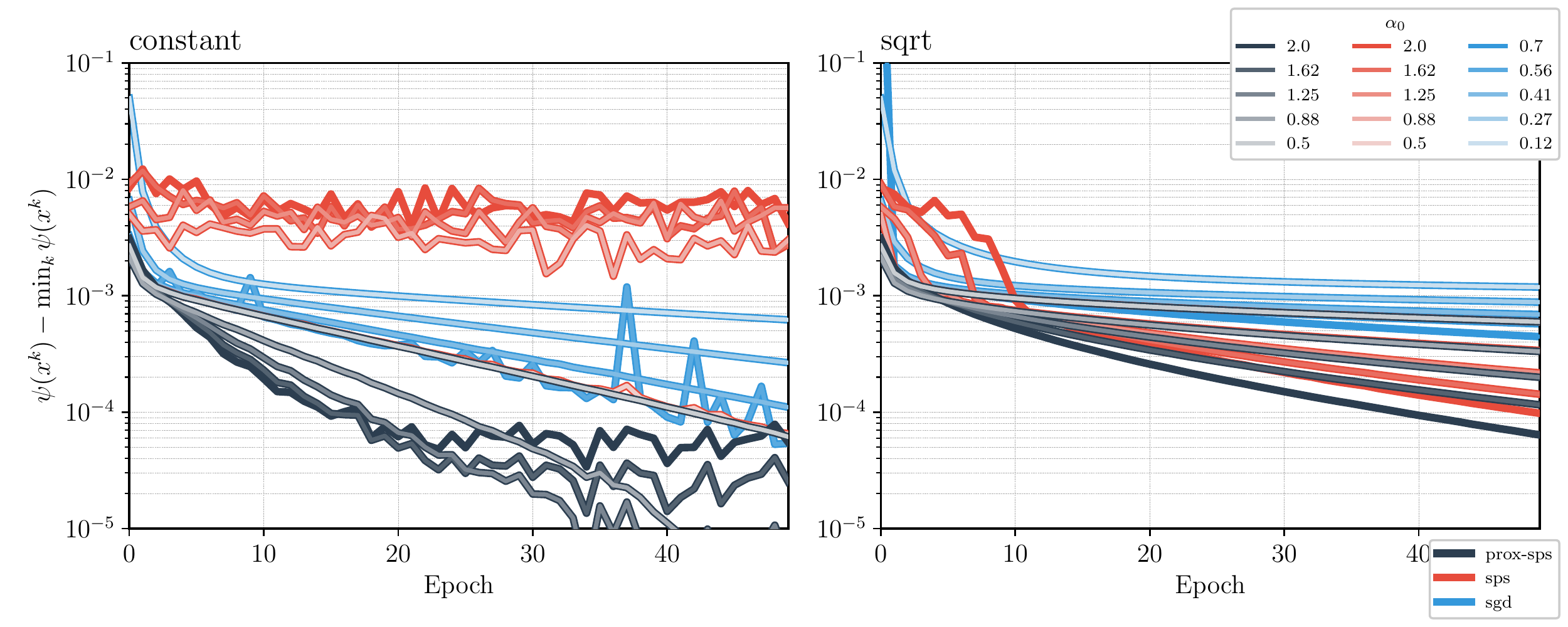}
    \caption{Objective function for  the Matrix Factorization problem~\eqref{prob-matrix-fac-reg}, with \texttt{constant} (left) and \texttt{sqrt} (right) step size schedule and several choices of initial values. Here $\min_k \psi(x^k)$ is the best objective function value found over all methods and all iterations.}
    \label{fig:matrix_fac2_stability}
\end{figure}
\begin{figure}[H]
    \centering
    \includegraphics[height=0.2\textheight]{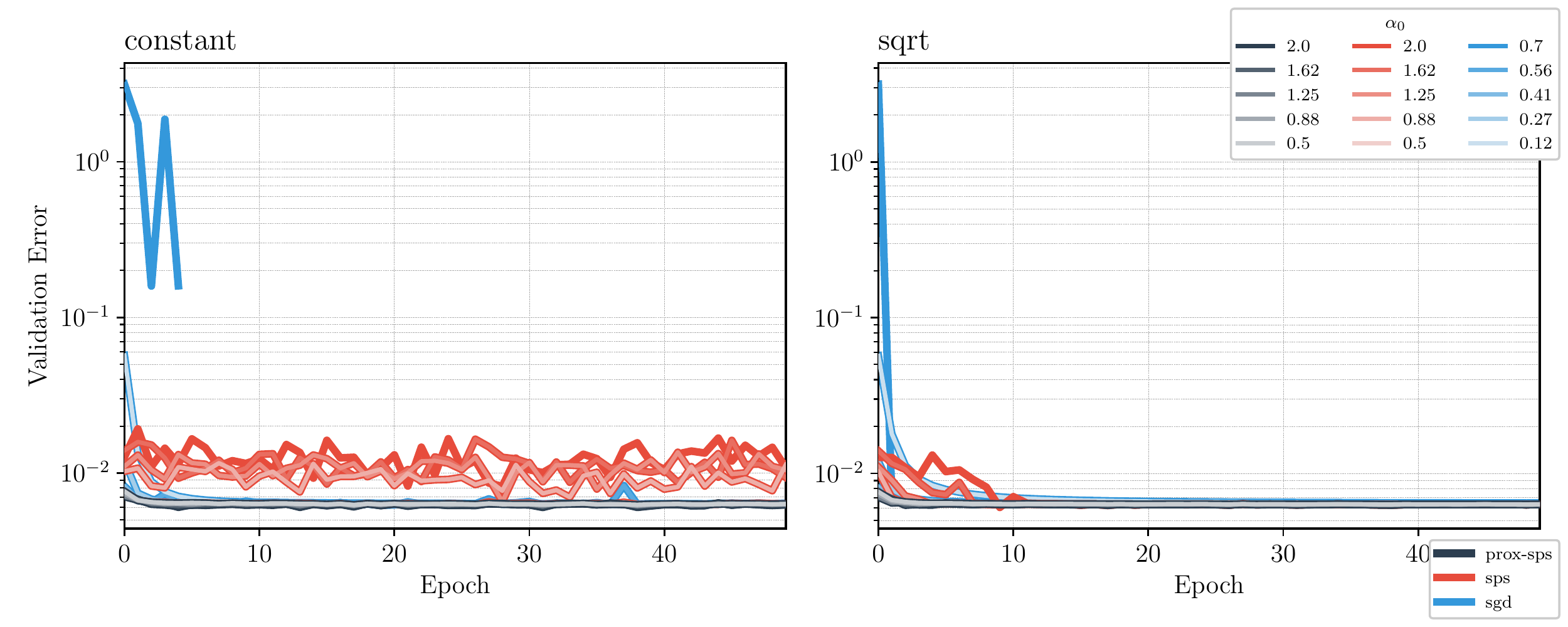}
    \caption{Validation error for   the Matrix Factorization problem~\eqref{prob-matrix-fac-reg}, with \texttt{constant} (left) and \texttt{sqrt} (right) step size schedule and several choices of initial values. }
    \label{fig:matrix_fac2_stability_val}
\end{figure}
\begin{figure}[H]
    \centering
    \includegraphics[width=0.99\textwidth]{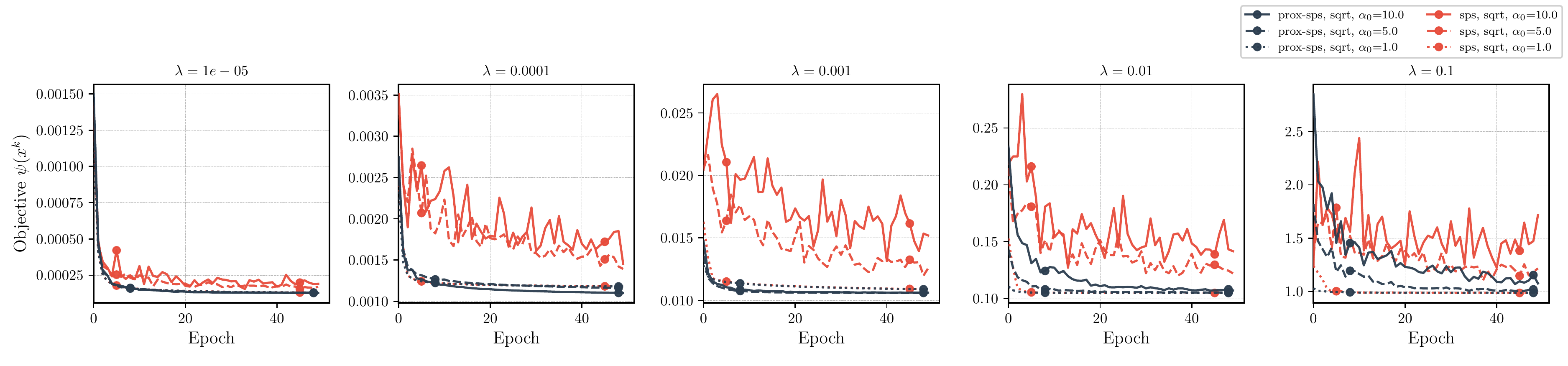}
    \includegraphics[width=0.99\textwidth]{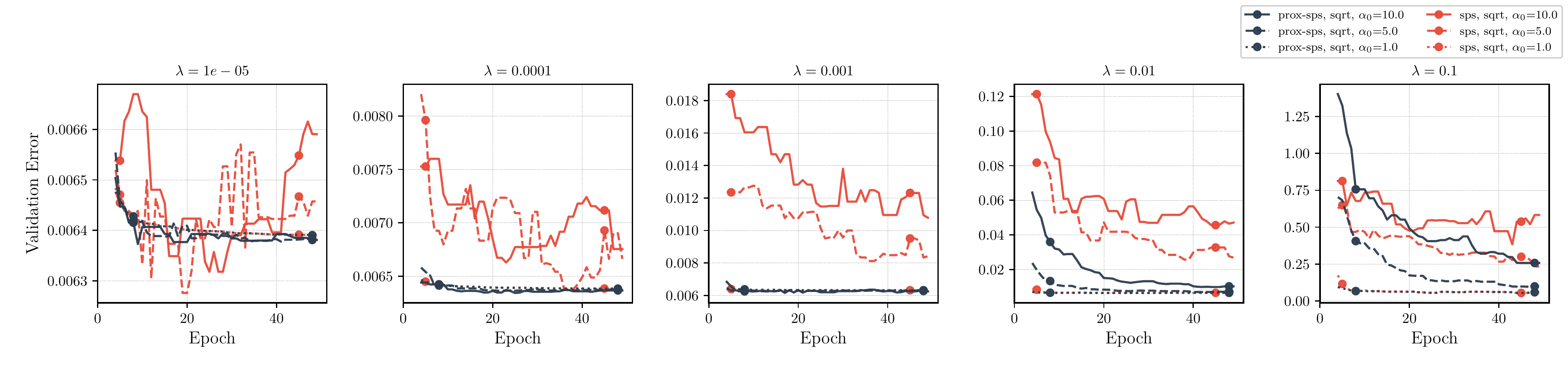}
    \caption{ Objective function value and validation error over the course of optimization. For the validation error, we plot a rolling median over five epochs in order to avoid clutter.}
    \label{fig:matrix_fac2}
\end{figure}

\subsection{Matrix completion experiment} \label{sec:exp-matrix-completion}

Consider an unknown matrix of interest $W\in \R^{d_1 \times d_2}$. Factorizing $W\approx U^\T V$ with $U\in \R^{r \times d_1},~V\in \R^{r \times d_2}$, we can estimate the entries of matrix $W$ as 
\begin{align}\label{eqn:model-matrix-complete}
    \hat{W}_{ij} = u_i^\T v_j+ b^U_{i} + b^V_{j}, \quad i \in [d_1], ~ j\in [d_2],
\end{align}
where $u_i$ is the $i$-th column of $U$ and $v_j$ is the $j$-th column of $V$, and $b^U \in \R^{d_1},~b^V \in \R^{d_2}$ are bias terms \citep{RiveraMunoz2022}.

We can interpret this as an empirical risk minimization problem as follows:
let $\mathcal{T}$ be the set of indices $(i,j)$ where $W_{ij}$ is known. 
With $\hat{W}_{ij}$ as in \eqref{eqn:model-matrix-complete} for trainable parameters $(U,V, b^U, b^V)$, the (regularized) problem is then given as 
\begin{align*}
    \min_{U,V, b^U, b^V} \frac{1}{|\mathcal{T}|} \sum_{(i,j) \in \mathcal{T}} (W_{ij} - \hat{W}_{ij})^2 + \frac{\lambda}{2}\|(U,V, b^U, b^V)\|^2.
\end{align*}

We use a dataset containing air quality measurements of a sensor network over one month. This dataset has been studied in \citet{RiveraMunoz2022}.\footnote{The dataset can be downloaded from \url{https://github.com/andresgiraldo3312/DMF/blob/main/DatosEliminados/Ventana_Eli_mes1.csv}.} The dataset contains measurements from 130 sensors over 720 timestamps, hence $d_1=130,~d_2=720$. In total, there are 56158 nonzero measurements (the rest was missing data or removed due to being an outlier). We split the nonzero measurements into a training set of size $|\mathcal{T}|=44926 \approx 0.8 \cdot 56158$ and the rest as a validation set. We standardize training and validation set using mean and variance of the training set. 
We set $r=24$ and use batch size 128. The validation error is defined as the root mean squared error on the elements of the validation set (which is not used for training). 

\textbf{Discussion}: The results are plotted in \cref{fig:sensor-metrics} and~\cref{fig:sensor-metrics2}. For all methods, we use a constant step size $\alpha_k$. \texttt{ProxSPS} achieves the smallest error on the validation set for the two smaller values of $\lambda$. For the largest $\lambda$, \texttt{ProxSPS}, \texttt{SPS} and \texttt{SGD} are almost identical for $\alpha_0 = 5$, but \texttt{SGD} with $\alpha_0=1$ is the best method. However, over all tested values of $\lambda$, \cref{fig:sensor-metrics2} shows that \texttt{ProxSPS} obtains the smallest error. Again, from the lower plot in \cref{fig:sensor-metrics} we can observe that \texttt{ProxSPS} produces iterates with smaller norm. 

\begin{figure}[t]
    \centering
    \includegraphics[trim=0 0.2cm 0 0.7cm, width=0.8\textwidth]{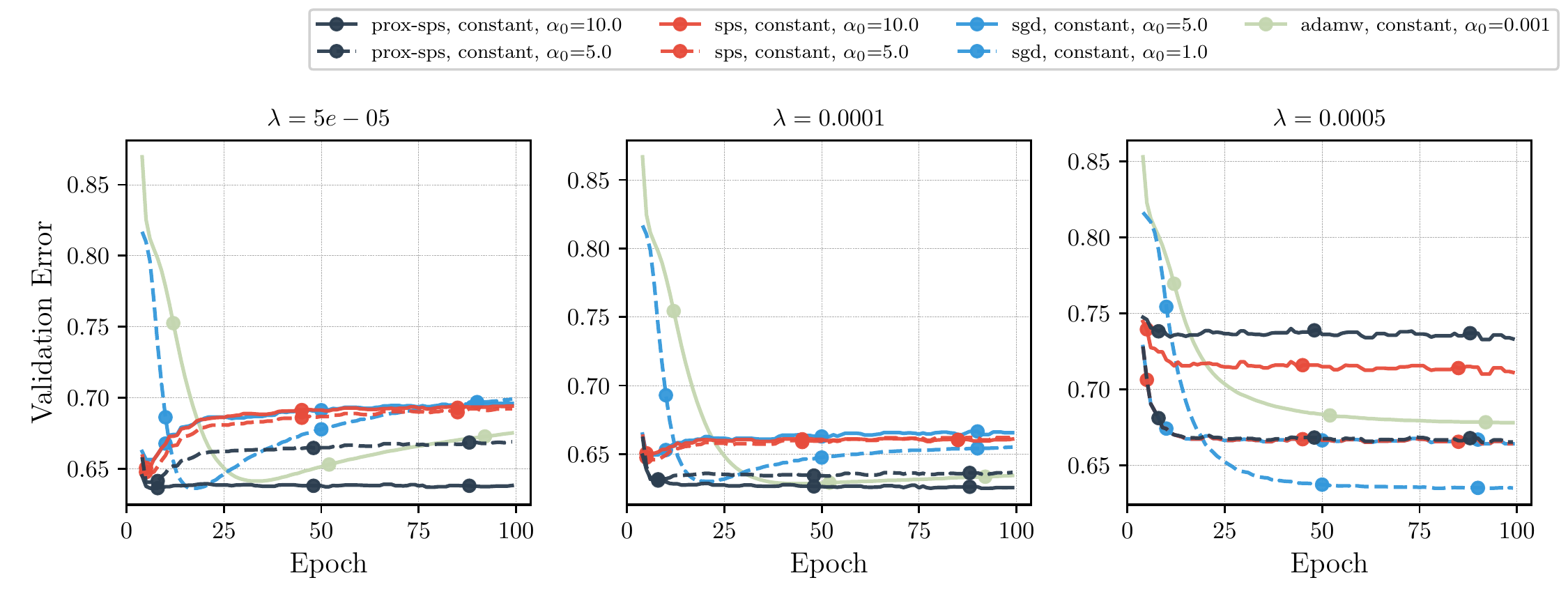}
    \includegraphics[trim=0 0.5cm 0 0.2cm, width=0.8\textwidth]{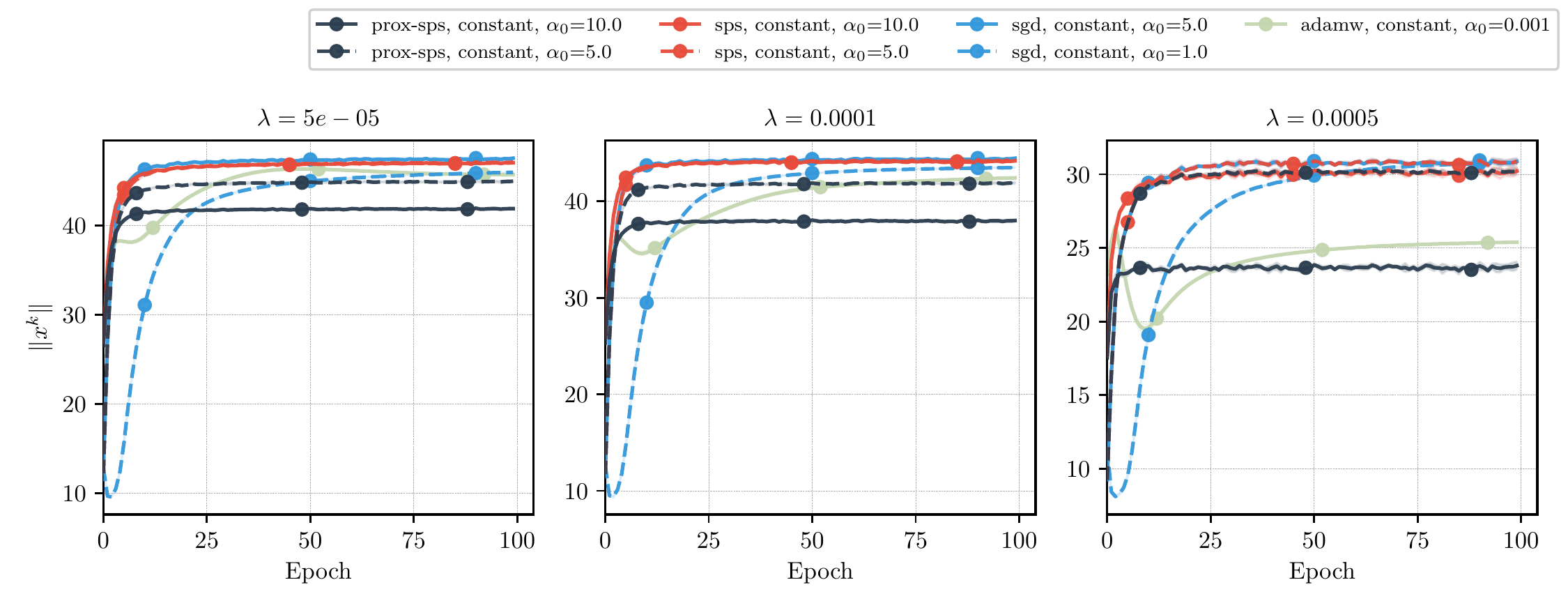}
    \caption[]
{\texttt{Matrix Completion}: Validation error (top) and model norm (top) for three values of the regularization parameter $\lambda$. Validation error is plotted as five-epoch running median. Shaded area is two standard deviations over ten independent runs.}
\label{fig:sensor-metrics}
\end{figure}
\begin{figure}[t]
    \centering
    \begin{subfigure}[t]{0.45\textwidth}
        \includegraphics[width=0.99\textwidth]{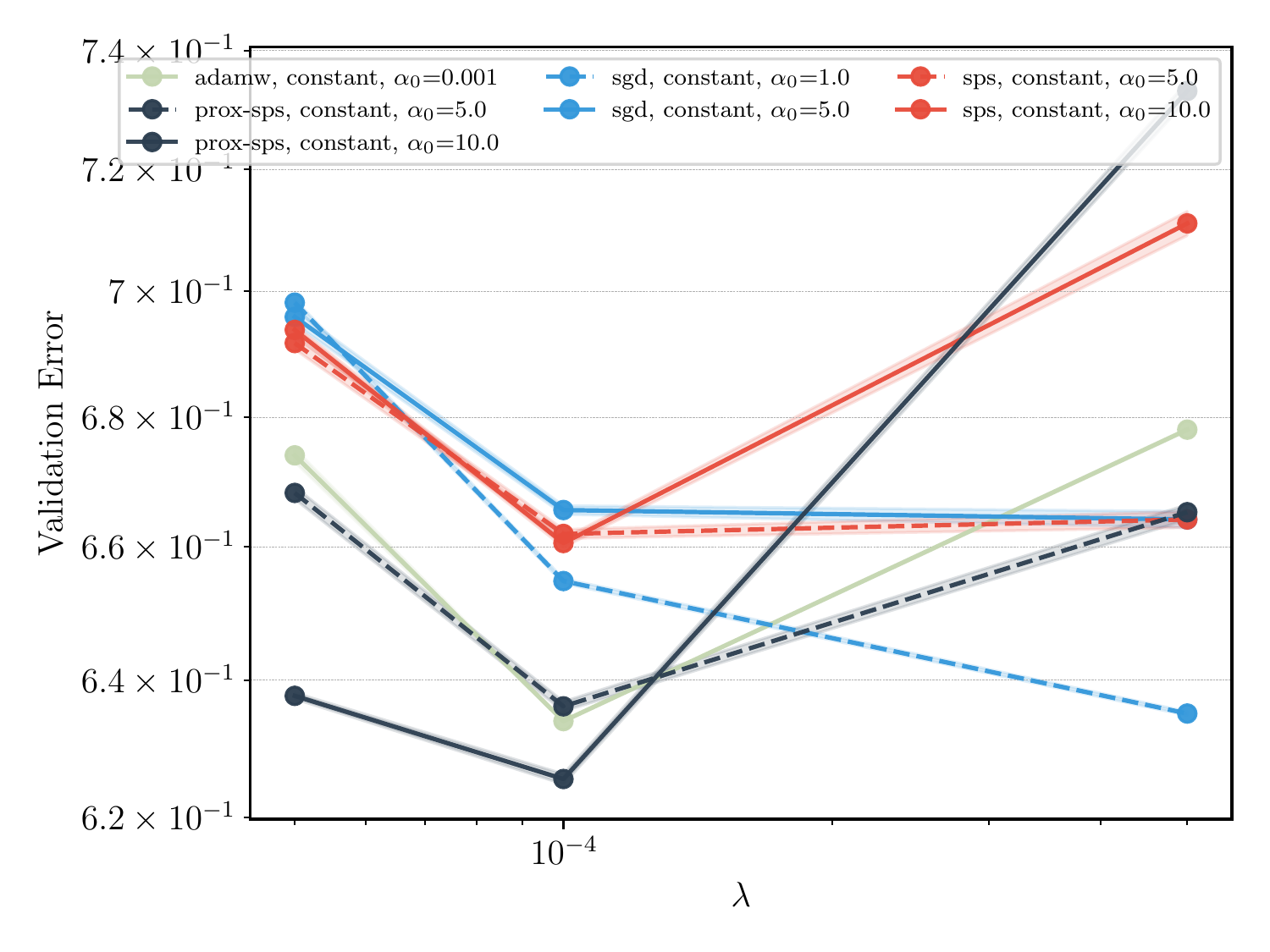}
        \caption{}
        \label{fig:sensor-metrics2}
    \end{subfigure}
    \begin{subfigure}[t]{0.45\textwidth}
        \includegraphics[width=0.99\textwidth]{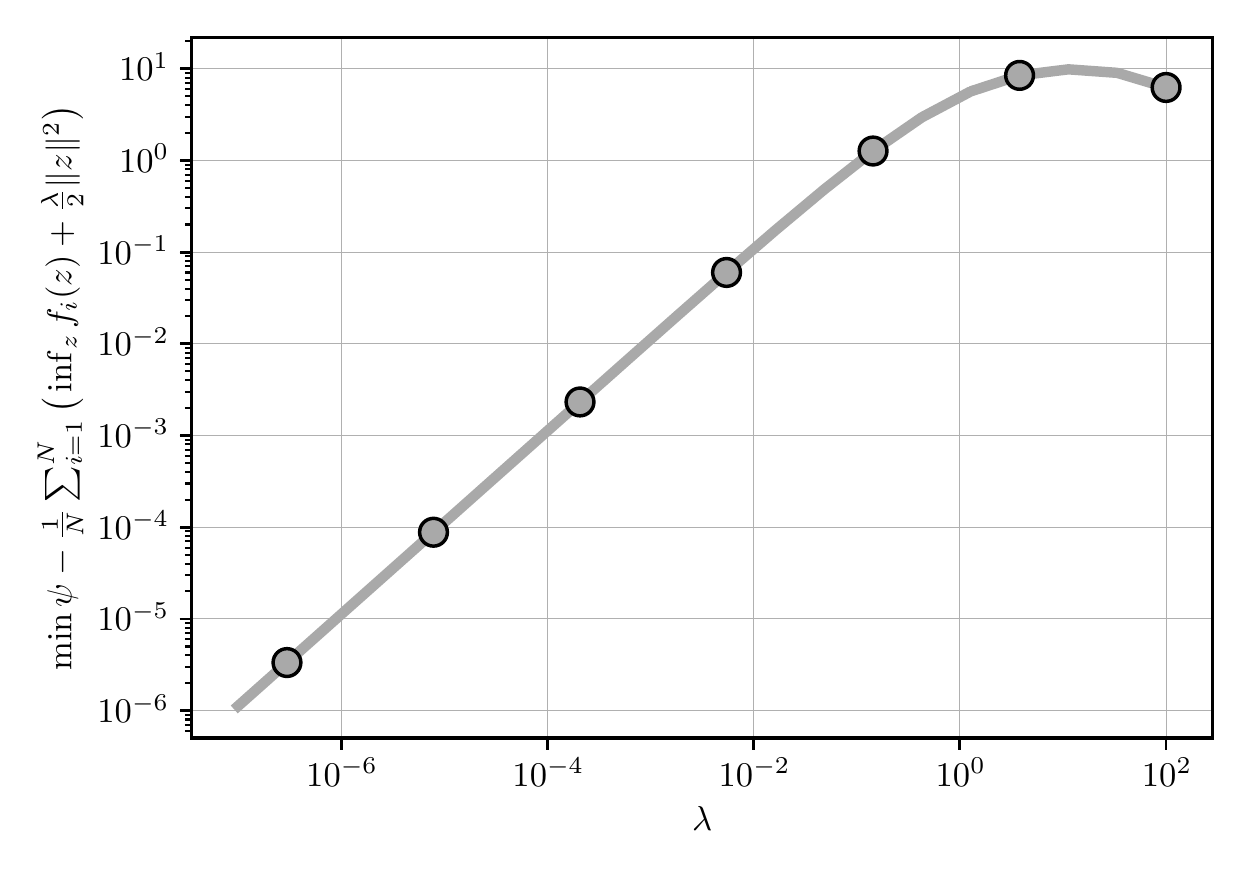}
        \caption{}
        \label{fig:interpolation-ridge}
    \end{subfigure}
    \caption{(a) \texttt{Matrix Completion}: Validation error as a function of the regularization parameter $\lambda$. Shaded area is one standard deviation (computed over ten independent runs). For all values, we take the median over epochs $[90,100]$.
    (b) Interpolation constant for a ridge regression problem for varying regularization parameter $\lambda$. See \cref{sec:interpolation-constant} for details.}
    \label{fig:additional-plots}
\end{figure}
\subsection{\texttt{Imagenet32} experiment} \label{sec:imagenet32-exp}

\texttt{Imagenet32} contains 1,28 million training and 50,000 test images of size $32\times32$, from 1,000 classes. We train the same \texttt{ResNet110} as described in \cref{sec:deep-learning-exp} with two differences: we exchange the output dimension of the final layer to 1,000 and activate batch norm. We use batch size 512. For this experiment we only run one repetition.

Similar to the setup in \cref{sec:deep-learning-exp}, we run all methods for three different values of $\lambda$. For \texttt{AdamW}, we use a constant learning rate $0.001$, for \texttt{SGD}, \texttt{SPS}, and \texttt{ProxSPS} we use the \texttt{sqrt}-schedule and $\alpha_0=1$. The validation accuracy and model norm are plotted in \cref{fig:imagenet32-resnet110-metrics}: we can observe that all methods perform similarly well in terms of accuracy. However, \texttt{AdamW} is more sensitive with respect to the choice of $\lambda$ and the norm of its iterates differs significantly from the other methods. Further, using an adaptive step size is advantageous: from \cref{fig:imagenet32-resnet110-step-sizes}, we see that the adaptive step size is active in the initial iterations, which leads to a faster learning of \texttt{(Prox)SPS} in the initial epochs compared to \texttt{SGD}. 

\begin{figure}[t]
    \centering
    \includegraphics[trim=0 0.2cm 0 0.7cm, width=0.8\textwidth]{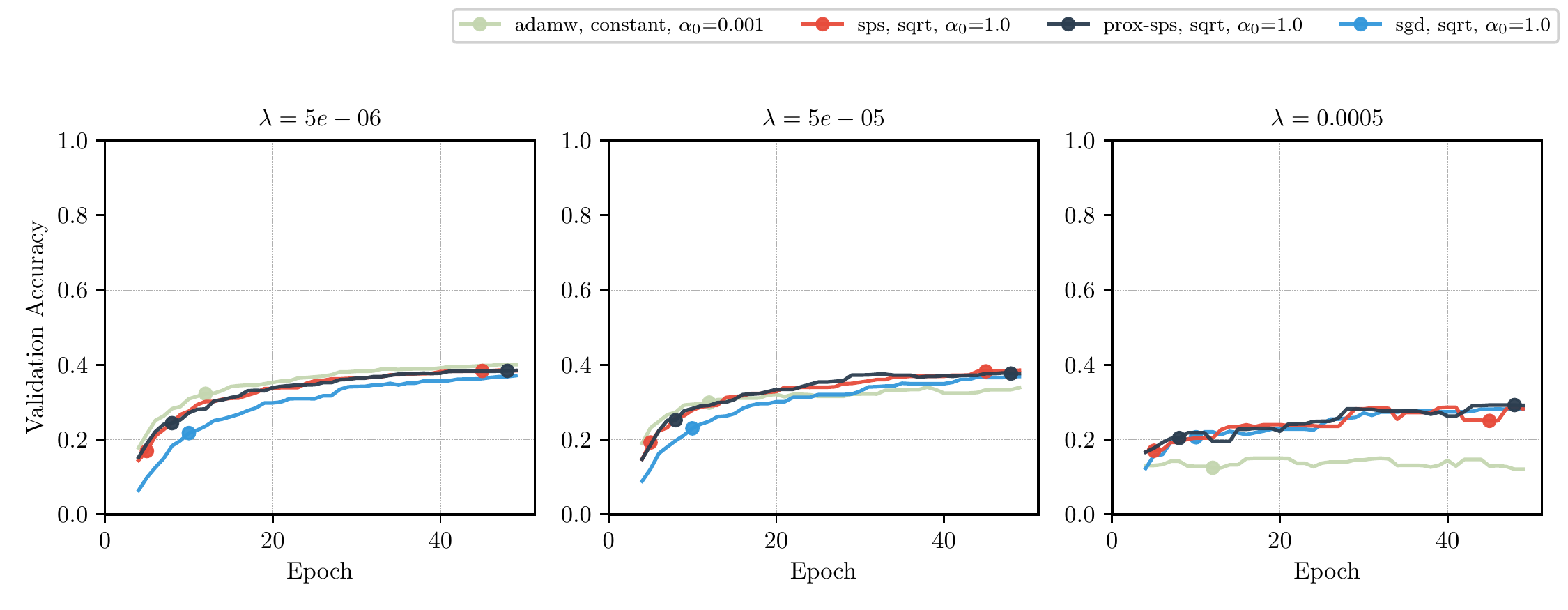}
    \includegraphics[trim=0 0.5cm 0 0.2cm, width=0.8\textwidth]{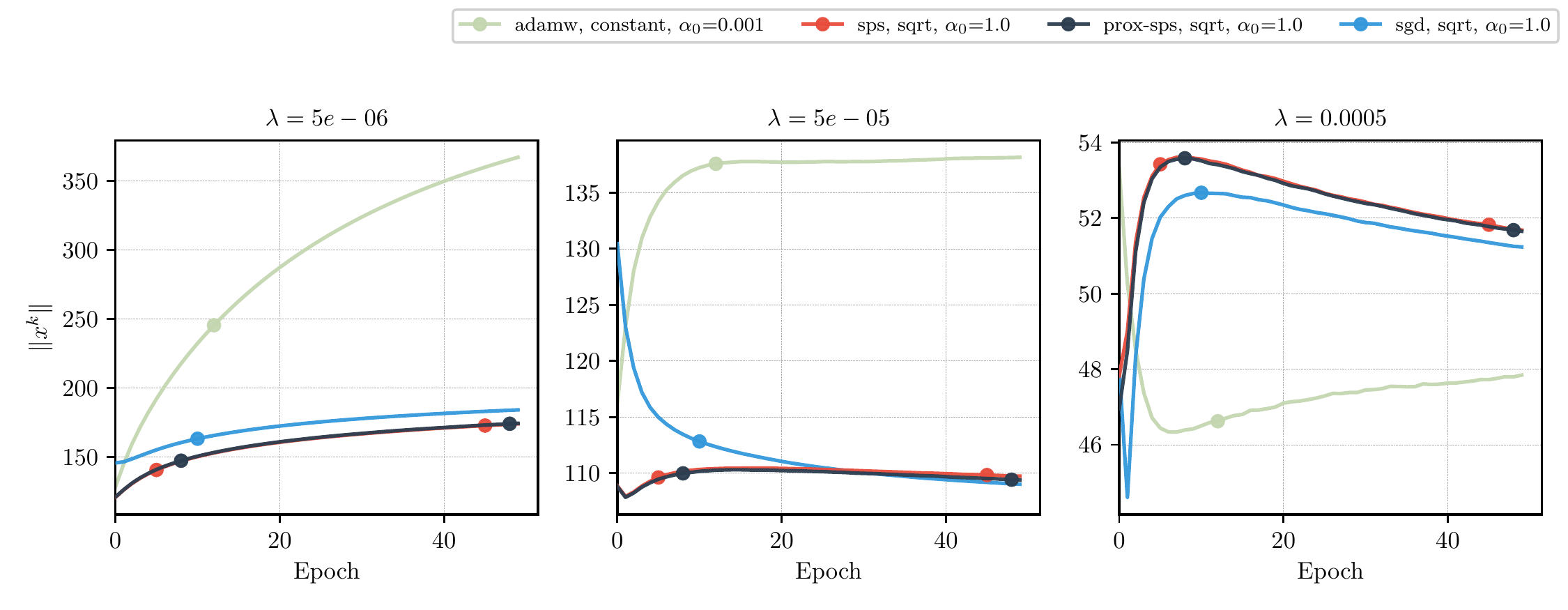}
    \caption[]
{\texttt{ResNet110} for \texttt{Imagenet32}: Validation accuracy as five-epoch running median (top) and model norm (bottom) for three values of $\lambda$.}
\label{fig:imagenet32-resnet110-metrics}
\end{figure}
\begin{figure}[t]
\centering
\includegraphics[width=0.85\textwidth]{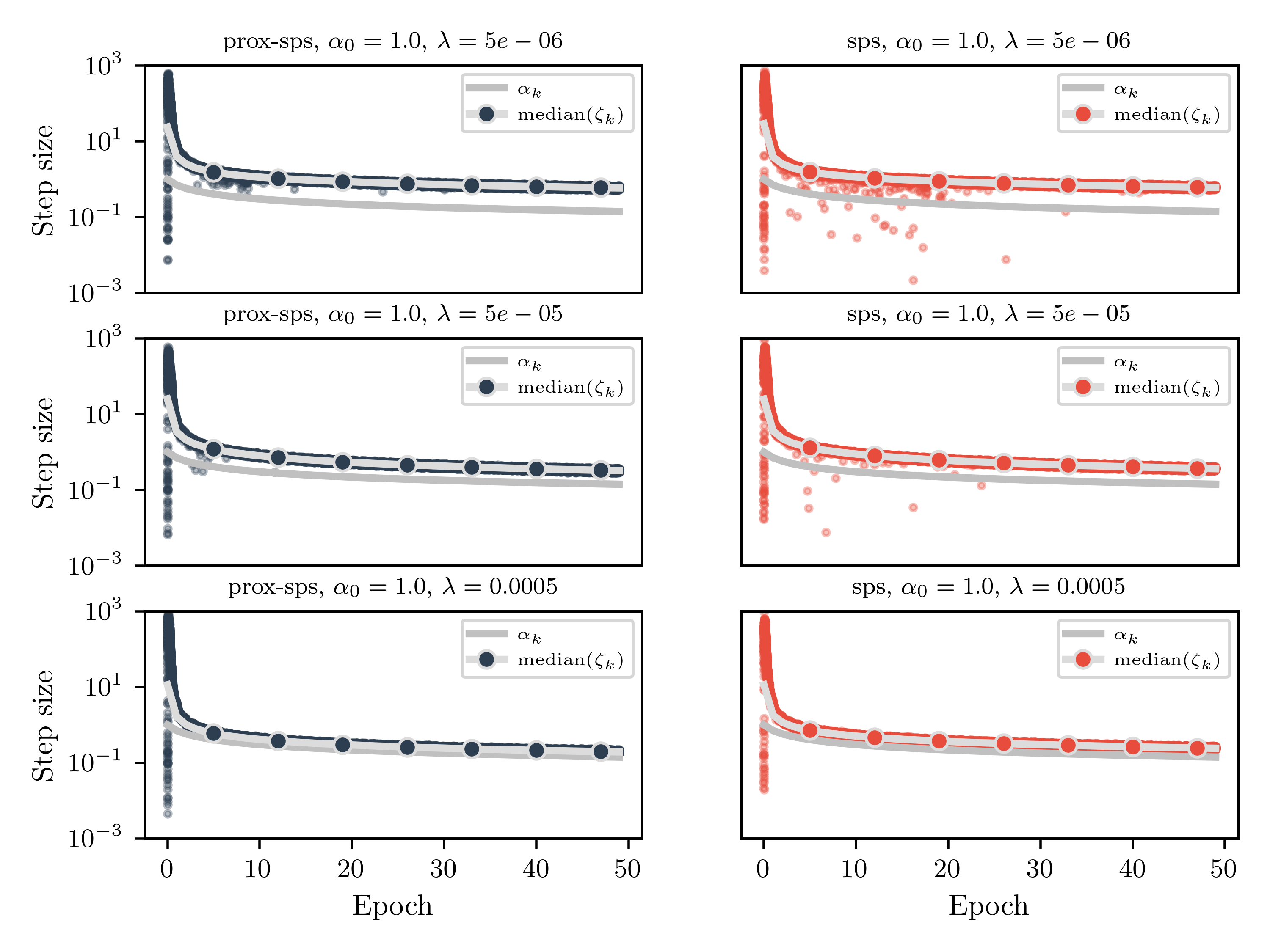}
\caption{\texttt{ResNet110} for \texttt{Imagenet32}: Adaptive step sizes for \texttt{SPS} and \texttt{ProxSPS}. See definition of $\zeta_k$ in \cref{sec:numerics-general}.}
\label{fig:imagenet32-resnet110-step-sizes}
\end{figure}

\subsection{Interpolation constant} \label{sec:interpolation-constant}

We illustrate how the interpolation constant $\sigma^2$ behaves if it would be computed for the regularized loss $\ell_i(x) = f_i(x) + \frac{\lambda}{2}\|x\|^2 $ (cf.\ also \cref{sec:sps-max-reg-naive}).
We do a simple ridge regression experiment. Let $A\in \R^{N\times n}$ be a matrix with row vectors $a_i \in \R^n,~i\in[N]$. We set $N=80,~n=100$ and generate $\hat x \in \R^n$ with entries drawn uniformly from $[0,1]$. We compute $b=A\hat x$. In this case, we have $f_i(x) = \frac12 (a_i^\top x - b_i)^2$ and $f(x) = \frac{1}{N}\sum_{i=1}^N f_i(x)$.

If one would apply the theory of \spsmax{} for the regularized loss functions $\ell_i$ with estimates $\underline{\ell}_i =0$, the constant $\sigma^2 = \big(\min_{x\in\R^n} f(x)+\phi(x)\big) - \frac{1}{N}\sum_{i=1}^N \inf_z \ell_i(z)$ determines the size of the constant term in the convergence results of \citep{Loizou2021,Orvieto2022}. We compute $\min_{x\in\R^n} f(x)+\phi(x)$ by solving the ridge regression problem. Further, the minimizer of $\ell_i$ is given by $(a_ia_i^\top + \lambda \mathbf{Id})^{-1}a_ib_i$. We plot $\sigma^2$ for varying $\lambda$ in \cref{fig:interpolation-ridge} to verify that $\sigma^2$ grows significantly if $\lambda$ becomes large (even if the loss could be interpolated perfectly, \ie $\inf_x f(x) = 0$).
%
%
We point out that the constant $\sigma^2$ does not appear in our convergence results \cref{thm:convex-smooth-reg} and \cref{thm:exact-nonconv-reg}. 
%
%
\end{document}